\documentclass[a4paper,12pt]{amsart}

\usepackage[top=3cm,left=2.5cm,right=2.5cm,bottom=3cm]{geometry}
\usepackage{relsize}
\usepackage{lineno}
\usepackage{hyperref}
\usepackage{amsmath,amsthm,pb-diagram,amssymb,comment}
\usepackage{amsfonts,graphicx,color}
\usepackage{enumitem, fancyhdr, dsfont}
\usepackage{thmtools,cleveref}
\usepackage[normalem]{ulem}
\usepackage{stmaryrd}
\usepackage{textcomp}
\usepackage{mathtools}
\usepackage{contour}
\usepackage{fontawesome5}
\usepackage{tikz,float}
\usepackage{chemformula}
\usepackage{xcolor} 
\newlist{steps}{enumerate}{1}
\setlist[steps, 1]{label = Step \arabic*:}
\usepackage{multicol}
\allowdisplaybreaks 

\hypersetup{
    colorlinks=true,
    linkcolor=dodger,
    filecolor=dodger,
    urlcolor=dodger,
    citecolor=dodger,
}

\setlist{topsep=0ex,itemsep=1ex}

    \newcommand{\thzfc}{\mathrm{ZFC}}

    \newcommand{\Awf}{\mathcal{A}}

    \newcommand{\Ewf}{\mathcal{E}}

    \newcommand{\Iwf}{\mathcal{I}}
    \newcommand{\Jwf}{\mathcal{J}}
    
    \newcommand{\Mwf}{\mathcal{M}}
    \newcommand{\Nwf}{\mathcal{N}}
    
    \newcommand{\Pwf}{\mathcal{P}}
    \newcommand{\Swf}{\mathcal{S}}

    \newcommand{\bfrak}{\mathfrak{b}}
    \newcommand{\cfrak}{\mathfrak{c}}
    \newcommand{\dfrak}{\mathfrak{d}}

    \newcommand{\menos}{\smallsetminus}
    
    \newcommand{\pts}{\mathcal{P}}
    
    \newcommand{\frestr}{{\upharpoonright}}

    \DeclareMathOperator{\add}{add}
    \DeclareMathOperator{\cov}{cov}

    \DeclareMathOperator{\non}{non}
    \DeclareMathOperator{\cof}{cof}

    \newcommand{\Dor}{\mathbb{D}}

    \newcommand{\Loc}{\mathbb{LOC}}

    \newcommand{\Por}{\mathbb{P}}
    
    \newcommand{\Qor}{\mathbb{Q}}
    \newcommand{\Ior}{\mathbb{I}}

    \newcommand{\Qnm}{\dot{\mathbb{Q}}}

    \newcommand{\SNwf}{\mathcal{SN}}

    \DeclareMathOperator{\cf}{cf}

    \newcommand{\imp}{\Rightarrow}

    \newcommand{\la}{\langle}
    \newcommand{\ra}{\rangle}

\newcommand{\supLc}{\mathrm{supLc}}
\newcommand{\minLc}{\mathrm{minLc}}

\newcommand{\Fn}{\mathrm{Fn}}

\newcommand{\Rbf}{\mathbf{R}}
\newcommand{\Cv}{\mathrm{Cv}}

\newcommand{\Lc}{\mathrm{Lc}}

\DeclareMathOperator{\Lb}{Lb}

\newcommand{\Scal}{\mathcal{S}}
\DeclareMathOperator{\id}{id}
\newcommand{\blc}{\mathfrak{b}^{\mathrm{Lc}}}
\newcommand{\dlc}{\mathfrak{d}^{\mathrm{Lc}}}

\newcommand{\vfa}{\mathfrak{v}}

\newcommand{\sqrb}{\sqsubset^{\bullet}}

\newcommand{\leqT}{\leq_{\mathrm{T}}}
\newcommand{\eqT}{\cong_{\mathrm{T}}}

\DeclareMathOperator{\minadd}{minadd}
\DeclareMathOperator{\mincof}{mincof}
\DeclareMathOperator{\mincov}{mincov}
\DeclareMathOperator{\minnon}{minnon}
\DeclareMathOperator{\supcof}{supcof}
\DeclareMathOperator{\supcov}{supcov}
\DeclareMathOperator{\supadd}{supadd}
\DeclareMathOperator{\supnon}{supnon}

\newcommand{\set}[2]{\{#1 \colon #2\}}

\newcommand{\seq}[2]{\la #1 \colon #2\ra}

\newcommand{\baire}{{}^{\omega}\omega}
\newcommand{\baireinc}{{}^{\uparrow\omega}\omega}
\newcommand{\cantor}{{}^{\omega}2}

\contourlength{0.8pt}
\newcommand{\lay}{\text{\ and\ }}

\newcommand{\NAwf}{\Nwf\!\Awf}
\newcommand{\MAwf}{\Mwf\Awf}
\newcommand{\IAwf}{\Iwf\Awf}
\newcommand{\EAwf}{\Ewf\Awf}

\contourlength{0.8pt}

\newcount\skewfactor
\def\mathunderaccent#1#2 {\let\theaccent#1\skewfactor#2
\mathpalette\putaccentunder}
\def\putaccentunder#1#2{\oalign{$#1#2$\crcr\hidewidth
\vbox to.2ex{\hbox{$#1\skew\skewfactor\theaccent{}$}\vss}\hidewidth}}

	\definecolor{ultramarineblue}{rgb}{0.25, 0.4, 0.96}
\definecolor{cornellred}{rgb}{0.7, 0.11, 0.11}
\definecolor{cobalt}{rgb}{0.0, 0.28, 0.67}
\definecolor{bleudefrance}{rgb}{0.19, 0.55, 0.91}
\definecolor{darkblue}{rgb}{0.0, 0.0, 0.55}
\definecolor{ferrarired}{rgb}{1.0, 0.11, 0.0}
\definecolor{brandeisblue}{rgb}{0.0, 0.44, 1.0}
\definecolor{azure(colorwheel)}{rgb}{0.0, 0.5, 1.0}
\definecolor{aqua}{rgb}{0.0, 1.0, 1.0}
\definecolor{aguamarina}{cmyk}{0.85,0,0.33,0}
\definecolor{cafe}{cmyk}{0,0.81,1,0.60}
\definecolor{canela}{cmyk}{0.14,0.42,0.56,0}
\definecolor{darkgray}{cmyk}{0,0,0,0.50}
\definecolor{emerald}{cmyk}{0.91,0,0.88,0.12}
\definecolor{fresa}{cmyk}{0,1,0.50,0}
\definecolor{gold}{cmyk}{0,0.10,0.84,0}
\definecolor{lightgray}{cmyk}{0,0,0,0.30}
\definecolor{marron}{cmyk}{0,0.72,1,0.45}
\definecolor{melon}{cmyk}{0,0.29,0.84,0}
\definecolor{ladri}{cmyk}{0,0.77,0.87,0}
\definecolor{olive}{cmyk}{0.64,0,0.95,0.40}
\definecolor{orange}{cmyk}{0,0.42,1,0}
\definecolor{peach}{cmyk}{0,0.46,0.50,0}
\definecolor{pink}{cmyk}{0,0.10,0.10,0}
\definecolor{orange}{cmyk}{0,0.42,1,0}
\definecolor{pine}{cmyk}{0.92,0,0.59,0.25}
\definecolor{purple}{cmyk}{0.45,0.86,0,0}
\definecolor{violet}{cmyk}{0.07,0.90,0,0.34}
\definecolor{craneorange}{RGB}{252,187,6}
\definecolor{red(ncs)}{rgb}{0.77, 0.01, 0.2}
\definecolor{caribbeangreen}{rgb}{0.0, 0.8, 0.6}

\definecolor{aguamarina}{cmyk}{0.85,0,0.33,0}
\definecolor{cafe}{cmyk}{0,0.81,1,0.60}
\definecolor{canela}{cmyk}{0.14,0.42,0.56,0}
\definecolor{darkgray}{cmyk}{0,0,0,0.50}
\definecolor{emerald}{cmyk}{0.91,0,0.88,0.12}
\definecolor{fresa}{cmyk}{0,1,0.50,0}
\definecolor{gold}{cmyk}{0,0.10,0.84,0}
\definecolor{lightgray}{cmyk}{0,0,0,0.30}
\definecolor{marron}{cmyk}{0,0.72,1,0.45}
\definecolor{melon}{cmyk}{0,0.29,0.84,0}
\definecolor{ladri}{cmyk}{0,0.77,0.87,0}
\definecolor{olive}{cmyk}{0.64,0,0.95,0.40}
\definecolor{orange}{cmyk}{0,0.42,1,0}
\definecolor{peach}{cmyk}{0,0.46,0.50,0}
\definecolor{pink}{cmyk}{0,0.10,0.10,0}
\definecolor{orange}{cmyk}{0,0.42,1,0}
\definecolor{pine}{cmyk}{0.92,0,0.59,0.25}
\definecolor{purple}{cmyk}{0.45,0.86,0,0}
\definecolor{violet}{cmyk}{0.07,0.90,0,0.34}

\DeclareSymbolFont{extraup}{U}{zavm}{m}{n}
\DeclareMathSymbol{\varheart}{\mathalpha}{extraup}{86}
\DeclareMathSymbol{\vardiamond}{\mathalpha}{extraup}{87}

\definecolor{dodger}{rgb}{0.0,0.5,1.0}
\newcommand{\dodger}[1]{{\color{dodger}#1}}

\definecolor{amber}{rgb}{1.0,0.49,0.0}

\definecolor{ogreen}{RGB}{107,142,35}

\title[Directed schemes of ideals and cardinal characteristics, I]{
Directed schemes of ideals and cardinal characteristics, I:\\ the meager additive ideal}

\author{Miguel A. Cardona}
\address{Faculty of Engineering, Instituci\'on Universitaria Pascual Bravo. Calle 73 No.~73A -- 226, Medell\'in, Colombia;
and\newline 
Einstein Institute of Mathematics, 
Edmond J. Safra Campus, Givat Ram\\
The Hebrew University of Jerusalem\\
Jerusalem, 91904, Israel}
\email{\href{mailto:miguel.cardona@pascualbravo.edu.co}{miguel.cardona@pascualbravo.edu.co}}
\urladdr{\url{https://sites.google.com/view/miacardonamo}}

\author{Diego A.~Mej\'ia}
\address{Graduate School of System Informatics, Kobe University. 1-1 Rokkodai-cho, Nada-ku, Kobe, Hyogo 657-8501 Japan}
\email{\href{mailto:damejiag@people.kobe-u.ac.jp}{damejiag@people.kobe-u.ac.jp}}
\urladdr{\url{https://researchmap.jp/mejia?lang=en}}

\author[I.E.~Rivera-Madrid]{Ismael E.\ Rivera-Madrid}
\address{Faculty of Engineering, Instituci\'on Universitaria Pascual Bravo. Calle 73 No. 73A - 226, Medell\'in, Colombia.}
\email{\href{mailto:ismael.rivera@pascualbravo.edu.co}{ismael.rivera@pascualbravo.edu.co}}



\thanks{The first author was partially supported by the Israel Science Foundation by grant 2320/23 (2023-2027); The second author was supported by the Grant-in-Aid for Scientific Research (C) 23K03198, Japan Society for the Promotion of Science; and the first and third authors were supported by the grant No.~AP0010, Direcci\'on de Tecnolog\'ia e Investigaci\'on and Oficina de Internacionalizaci\'on, Instituci\'on Universitaria Pascual Bravo}

\definecolor{sub0}{RGB}{29,32,137}
\definecolor{sub1}{RGB}{1,71,157}
\definecolor{sub2}{RGB}{1,104,183}
\definecolor{sub3}{RGB}{0,160,234}
\definecolor{sug}{RGB}{0,154,68}
\definecolor{suy}{RGB}{208,219,1}

\DeclareUnicodeCharacter{0301}{\'{e}}

\begin{document}

\makeatletter
\def\@roman#1{\romannumeral #1}
\newcommand{\startlist}{\ \@beginparpenalty=10000}
\makeatother

\newcounter{enuAlph}
\renewcommand{\theenuAlph}{\Alph{enuAlph}}

\numberwithin{equation}{section}
\renewcommand{\theequation}{\thesection.\arabic{equation}}


\theoremstyle{plain}
  \newtheorem{theorem}[equation]{Theorem}
  \newtheorem{corollary}[equation]{Corollary}
  \newtheorem{lemma}[equation]{Lemma}
  \newtheorem{mainlemma}[equation]{Main Lemma}
  \newtheorem*{mainthm}{Main Theorem}
  \newtheorem{prop}[equation]{Proposition}
  \newtheorem{clm}[equation]{Claim}
  \newtheorem{subclm}[equation]{Subclaim}
  \newtheorem{fact}[equation]{Fact}
  \newtheorem{exer}[equation]{Exercise}
  \newtheorem{question}[equation]{Question}
  \newtheorem{aim}[equation]{Aim}
  \newtheorem{problem}[equation]{Problem}
  \newtheorem{conjecture}[equation]{Conjecture}
  \newtheorem{assumption}[equation]{Assumption}
    \newtheorem{hopethm}[equation]{Hopeful Theorem}
    \newtheorem{challenging}[enuAlph]{Main challenging}
    \newtheorem{hopele}[equation]{Hopeful Lemma}
    \newtheorem{discussion}[equation]{Discussion}
  \newtheorem*{thm}{Theorem}
  \newtheorem{teorema}[enuAlph]{Theorem}
  \newtheorem*{corolario}{Corollary}
\theoremstyle{definition}
  \newtheorem{definition}[equation]{Definition}
  \newtheorem{example}[equation]{Example}
  \newtheorem{remark}[equation]{Remark}
  \newtheorem{notation}[equation]{Notation}
  \newtheorem{context}[equation]{Context}

  \newtheorem*{defi}{Definition}
  \newtheorem*{acknowledgements}{Acknowledgements}

\def\sectionautorefname{Section}
\def\subsectionautorefname{Subsection}


\begin{abstract}
We introduce the notion of \emph{directed scheme of ideals} to characterize peculiar ideals on the reals, which comes from a formalization of the framework of Yorioka ideals for strong measure zero sets. We prove general theorems for directed schemes and propose a directed scheme $\vec{\mathcal{M}} = \{\mathcal{M}_I \colon I\in\mathbb{I}\}$ for the ideal $\mathcal{MA}$ of meager-additive sets of reals. This directed scheme does not only helps us to understand more the combinatorics of $\mathcal{MA}$ and its cardinal characteristics, but provides us new characterizations of the additivity and cofinality numbers of the meager ideal of the reals. 


In addition, we display connections between the characteristics associated with $\Mwf_I$ and other classical characteristics. Furthermore, we demonstrate the consistency of $\mathrm{cov}(\mathcal{N}\!\mathcal{A})<\mathfrak{c}$ and $\mathrm{cof}(\mathcal{MA})<\mathrm{non}(\mathcal{SN})$. The first one answers a question raised by the authors in~\cite{CMR2}. 
\end{abstract}
\maketitle


\makeatother

\section{Introduction}\label{s0}

Recent study of $\SNwf$, the ideal of strong measure zero subsets of $\cantor$, and its cardinal characteristics, involves its characterization via the \emph{Yorioka ideals}. These have been used to control de cofinality of $\SNwf$ in forcing models~\cite{Yorioka,CM,cardona,CMRM,CarMej23}, as well as its additivity number~\cite{BCM2}. These ideals have helped to overcome challenges to control the combinatorics of $\SNwf$ and its cardinal characteristics, particularly in forcing models.

We aim to generalize the framework of Yorioka ideals to control de combinatorics of other very peculiar ideals. In this paper, we give special attention to the ideal of meager-additive subsets of $\cantor$.  


\begin{notation}\label{nt:1}
Before proceeding, we first fix some notation.
\begin{enumerate}[label = \normalfont(\arabic*)]
    \item\label{nt:1:1} Given a formula $\phi$, $\forall^\infty n<\omega\colon \phi$ means that all but finitely many natural numbers satisfy $\phi$; $\exists^\infty n<\omega\colon \phi$ means that infinitely many natural numbers satisfy $\phi$.

    \item\label{nt:1:2}  $\baireinc$ denotes the set of all increasing functions in $\baire$.

    \item\label{nt:1:3} For $x,y\in\baire$, $x\leq^* y$ means that $\forall^\infty n<\omega\colon x(i)\leq y(i)$.

    \item The identity function on $\omega$ is denoted by $\id$. 

    \item\label{nt:1:4} $\cfrak:=2^{\aleph_0}$; $\Nwf$ and $\Mwf$ denote the ideals of Lebesgue measure zero sets and of meager sets in the Cantor space  $\cantor$, respectively. Let $\Ewf$ be the $\sigma$-ideal generated by the closed measure zero subsets of $\cantor$. It is well-known that $\Ewf\subseteq\Nwf\cap\Mwf$. Even more, it was proved that $\Ewf$ is a proper subideal of $\Nwf\cap\Mwf$ (see~\cite[Lemma 2.6.1]{bartJudah}).  

    \item\label{nt:1:5} A \emph{preorder} is a pair $\la P,\leq\ra$ where $P$ is a set and $\leq$ is a reflexive and transitive relation on $P$.

    \item\label{nt:1:6} A preorder $\la S,\leq\ra$ is \emph{directed} if, for $s,s'\in S$, there is some $t\in S$ such that $s\leq t$ and $s'\leq t$. 

    \item\label{nt:1:7} $\Ior$ denotes the set of partitions $I=\seq{I_n}{n<\omega}$ of $\omega$ into consecutive non-empty intervals. We consider the directer preorder on $\Ior$ defined by $I\sqsubseteq J$ iff $\forall^\infty n<\omega\, \exists k<\omega\colon I_k\subseteq J_n$. 
\end{enumerate}
Given an ideal $\Iwf$ of subsets of $X$ such that $\{x\}\in \Iwf$ for all $x\in X$, \emph{the cardinal characteristics associated with $\Iwf$} are defined by
\begin{align*}
 \add(\Iwf)&:=\min\set{|\Jwf|}{\Jwf\subseteq\Iwf\text{\ and } \bigcup\Jwf\notin\Iwf};\\
     \cov(\Iwf)&:=\min\set{|\Jwf|}{\Jwf\subseteq\Iwf\text{\ and }\bigcup\Jwf=X};\\
    \non(\Iwf)&:=\min\set{|A|}{A\subseteq X\text{\ and }A\notin\Iwf};\\
     \cof(\Iwf)&:=\min\set{|\Jwf|}{\Jwf\subseteq\Iwf\text{\ is cofinal in }\la\Iwf,\subseteq\ra}.   
\end{align*}
These cardinals are called \emph{additivity, covering, uniformity, and cofinality of $\Iwf$}, respectively. The relationship between the cardinals defined above is illustrated in \autoref{diag:idealI}.
\end{notation}

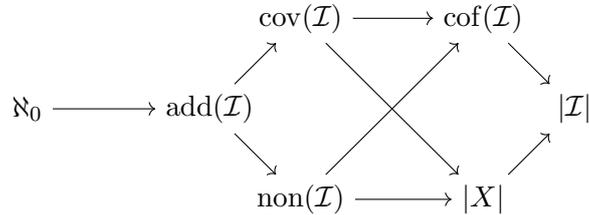
\begin{figure}[ht!]
\centering
\begin{tikzpicture}[scale=1.2]
\small{
\node (azero) at (-1,1) {$\aleph_0$};
\node (addI) at (1,1) {$\add(\Iwf)$};
\node (covI) at (2,2) {$\cov(\Iwf)$};
\node (nonI) at (2,0) {$\non(\Iwf)$};
\node (cofI) at (4,2) {$\cof(\Iwf)$};
\node (sizX) at (4,0) {$|X|$};
\node (sizI) at (5,1) {$|\Iwf|$};

\draw (azero) edge[->] (addI);
\draw (addI) edge[->] (covI);
\draw (addI) edge[->] (nonI);
\draw (covI) edge[->] (sizX);
\draw (nonI) edge[->] (sizX);
\draw (covI) edge[->] (cofI);
\draw (nonI) edge[->] (cofI);
\draw (sizX) edge[->] (sizI);
\draw (cofI) edge[->] (sizI);
}
\end{tikzpicture}
\caption{Diagram of the cardinal characteristics associated with $\Iwf$. An arrow  $\mathfrak x\rightarrow\mathfrak y$ means that (provably in ZFC) 
    $\mathfrak x\le\mathfrak y$.}
\label{diag:idealI}
\end{figure}

Yorioka~\cite{Yorioka} defined ideals $\Iwf_f$ on $\cantor$, parametrized by $f\in\baireinc$, which have Borel bases and characterize $\SNwf:=\bigcap_{f\in\baireinc}\Iwf_f$. A further property of the Yorioka ideals is that $\Iwf_g\subseteq \Iwf_f$ whenever $f\leq ^* g$. 

These ideals led Yorioka to rediscover and significantly improve Seredy\'nski \cite{Sered} result about the cofinality of $\SNwf$, and to show that no inequality between $\cof(\SNwf)$ and $\cfrak$ can be decided in $\thzfc$. More research about the cofinality has been conducted by the first and second authors~\cite{cardona, CM23}. One representative result is the following, whose notation is clarified afterward.

\begin{theorem}[{\cite[Thm.~4.6]{CM23}}]\label{thm:a12}
$\cof(\SNwf)\leq\cov\left(\left([\supcof(\vec\Swf)]^{<\minadd(\vec\Swf)}\right)^\dfrak\right).$
\end{theorem}

We discovered that the result above is not associated with the meaning of \emph{strong measure zero} but of the framework given by the Yorioka ideals. We generalize this framework below, with the central notion of this paper.

\begin{definition}\label{def:ds}
Let $X$ be a set and let $S$ be a directed preorder. We say that a sequence $\vec\Iwf=\seq{\Iwf_s}{s\in S}$ of ideals on $X$ is a \emph{directed scheme of ideals on $X$} if $\Iwf_{s'}\subseteq \Iwf_s$ whether $s\leq s'$. Put $\vec{\Iwf}^*=\bigcap_{s\in S}\Iwf_s$, which we simply abbreviate $\Iwf^*$ when clear from the context. 

These directed schemes of ideals give rise to the study of the cardinals associated with $\Iwf_s$. Moreover, we associate the following cardinal characteristics to $\vec{\Iwf}$. 
\begin{align*}
    \minadd(\vec{\Iwf}) & := \min_{s\in S}\add(\Iwf_s), & \supadd(\vec{\Iwf}) & := \sup_{s\in S}\add(\Iwf_s),\\
    \mincov(\vec{\Iwf}) & := \min_{s\in S}\cov(\Iwf_s), & \supcov(\vec{\Iwf}) & := \sup_{s\in S}\cov(\Iwf_s),\\
    \minnon(\vec{\Iwf}) & := \min_{s\in S}\non(\Iwf_s), & \supnon(\vec{\Iwf}) & := \sup_{s\in S}\non(\Iwf_s),\\
    \mincof(\vec{\Iwf}) & := \min_{s\in S}\cof(\Iwf_s). & \supcof(\vec{\Iwf}) & := \sup_{s\in S}\cof(\Iwf_s).
\end{align*}
\end{definition}

It is clear that $\cov(\Iwf_s) \leq \cov(\Iwf_{s'})$ and $\non(\Iwf_{s'})\leq\non(\Iwf_s)$ whenever $s\leq s'$ in $S$. Then, whenever $S$ has a minimum element $s_0$, $\mincov(\vec\Iwf) = \cov(\Iwf_{s_0})$ and $\supnon(\vec\Iwf) = \non(\Iwf_{s_0})$.

For example, the Yorioka ideals form a directed scheme $\vec\Swf:=\seq{\Iwf_f}{f\in\baireinc}$ where $\baireinc$ is ordered by $\leq^*$, so $\SNwf = \vec\Swf^*$. Since $\id$ is a $\leq^*$ minimum element of $\baireinc$, $\mincov(\vec\Swf) = \cov(\Iwf_{\id})$ and $\supnon(\vec\Swf) = \non(\Iwf_{\id})$. In this context, we also have $\supadd(\vec\Swf) = \add(\Iwf_{\id})$ and $\mincof(\vec\Swf) = \cof(\Iwf_{\id})$, see~\cite[Thm.~3.15]{CM}.

Directed schemes of ideals is a crucial notion when it comes to obtaining information on the cardinals of $\Iwf^*$ . In fact, the two first authors~\cite{cardona,CM23} used $\vec\Swf$ when studying the cofinality of $\SNwf$. Moreover, in \cite{CM}, they used these ideals to investigate the rest of the cardinals associated with $\SNwf$.  So one of the aims of this paper is to generalize most of the results established of the cardinals of $\SNwf$ via Yorioka ideals in the context of the directed scheme of ideals. For example, \autoref{thm:a12} is a consequence of our following main result, which we prove in \autoref{sec:s1}.

\begin{teorema}\label{thm:a10}
 Let $\vec\Iwf$ be a directed scheme of ideals on $S$. Then:
 \begin{enumerate}[label=\rm(\alph*)]
     \item $\minadd(\vec{\Iwf})\leq\add(\Iwf^*)$;
     \item $\supcov(\vec{\Iwf})\leq\cov(\Iwf^*) \leq \cov\left(\prod_{s\in D}\Cv_{\Iwf_s}\right)$;
     \item $\non(\Iwf^*)=\minnon(\vec{\Iwf})$;
      \item $\cof(\Iwf^*)\leq\dfrak\left(\prod\vec\Iwf\right)\leq\cov\left(\left([\supcof(\vec{\Iwf})]^{<\minadd(\vec\Iwf)}\right)^{\cof(S)}\right)$.
 \end{enumerate}
\end{teorema}

Throughout time, numerous characterizations of $\SNwf$ were discovered. One of the very interesting characterizations of $\SNwf$ is the Galvin–Mycielski–Solovay Theorem stated below, considering $\cantor$ as a topological group with the standard coordinate-wise addition modulo $2$.

\begin{theorem}[{\cite{GaMS}}]
A set $X\subseteq\cantor$ is $\SNwf$ if and only if $X+M\neq\cantor$ for each $M\in\Mwf$.
\end{theorem}

This theorem led to the rise of a certain kind of small sets, such as the meager-additive and null-additive sets, which are defined as follows.

\begin{definition}
Let $\Iwf\subseteq\Pwf(\cantor)$ be an ideal.
\begin{enumerate}
    \item $\Iwf$ is 
\emph{translation invariant} if $A+x\in\Iwf$ for each $A\in\Iwf$ and $x\in\cantor$
    \item A set $X\subseteq\cantor$ is termed \emph{$\Iwf$-additive} if, for every $A\in\Iwf$, $A+X\in\Iwf$. Denote by $\IAwf$ the collection of the $\Iwf$-additive subsets of $\cantor$. Notice that $\IAwf$ is a ($\sigma$-)ideal and $\IAwf\subseteq\Iwf$ when $\Iwf$ is a translation invariant ($\sigma$-)ideal. 
    \item The members of $\MAwf$ are called \emph{meager-additive}, while those in $\NAwf$ are called \emph{null-additive}.
\end{enumerate} 
\end{definition}

The cardinal characteristics associated with $\IAwf$ and $\Iwf$ are easily related as follows.

\begin{lemma}[e.g.~{\cite[Lem.~1.3]{CMR2}}]\label{lem:a0}
For any translation invariant ideal $\Iwf$ on $\cantor$: 
\begin{multicols}{3}
\begin{enumerate}[label=\rm(\alph*)]
    \item\label{lem:a0:1} $\add(\Iwf)\leq\add(\IAwf)$.

    \item\label{lem:a0:2} $\non(\IAwf)\leq\non(\Iwf)$.

    \item\label{lem:a0:3} $\cov(\Iwf)\leq\cov(\IAwf)$.
\end{enumerate}
\end{multicols}
\end{lemma}

The ideal $\IAwf$ has attracted a lot of attention when $\Iwf$ is either $\Mwf$ or $\Nwf$. For instance, they were investigated by Bartoszy\'nski and Judah~\cite{bartJudah}, Pawlikowski~\cite{paw85}, Shelah~\cite{shmn},  Zindulka~\cite{zin19}, and more recently by the authors~\cite{CMR2}.

The ideals $\MAwf$ and $\NAwf$ are characterized below: 

\begin{theorem}\label{thm:a0} Let $X\subseteq\cantor$.
    \begin{enumerate}[label=\rm(\arabic*)]
        \item \emph{(\cite[Thm.~13]{shmn})}\label{thm:a0:NA}
 $X\in\NAwf$ iff for all $I=\seq{I_n}{n\in\omega}\in\Ior$ there is some $\varphi\in\prod_{n\in \omega}\pts({}^{I_n}2)$ such that $\forall n\in \omega\colon |\varphi(n)|\leq n$ and $X\subseteq H_\varphi$,  where \[
H_\varphi:=\set{x\in\cantor}{\forall^{\infty} n\in \omega\colon x{\upharpoonright}I_n\in \varphi(n)}.
\]
        \item\emph{(\cite[Thm.~2.2]{bartJudah})}\label{thm:a0:MA}
 $X\in\MAwf$ iff for all $I\in\Ior$ there are $J\in\Ior$ and $y\in\cantor$  such that  \[\tag{\faLeaf}\label{for:ao} \forall x\in X\, \forall^\infty n<\omega\, \exists k<\omega\colon I_k\subseteq J_n\text{\ and\ }x{\upharpoonright}I_k=y{\upharpoonright}I_k.\] 
 Moreover, Shelah~\cite[Thm.~18]{shmn} proved that $J$ can be found coarser than $I$, i.e.\ every member of $J$ is the union of members of $I$
    \end{enumerate}
\end{theorem}
These characterizations led Shelah~\cite{shmn} to prove that every null-additive set is meager-additive, that is $\NAwf\subseteq\MAwf$. Zindulka~\cite{zin}, on the other hand, used combinatorial properties
of meager-additive sets described by Shelah and Pawlikowski to
characterize meager-additive sets in $\cantor$ in a way that nicely parallels the definition of strong measure zero sets. This led him to establish the following:

\begin{theorem}\label{thm:a1}
 $\EAwf=\MAwf$.   
\end{theorem}

In the present paper, motivated by~\eqref{for:ao}, we offer a directed scheme of ideals on $\cantor$ for the ideal of
meager-additive sets, which we denote $\vec\Mwf=\seq{\Mwf_I}{I\in\Ior}$ (where $\Ior$ is ordered by $\sqsubseteq$), i.e. 
\[\MAwf=\bigcap\set{\Mwf_I}{I\in\Ior}.\]

In terms of the ideal $\Mwf_{I}$, we also provide a reformulation of $\Mwf$. Concretely, we have $\Mwf_{I^1}=\Mwf$ where $I^1$ denotes the partition of $\omega$ into singletons, which is a $\sqsubseteq$-minimum element of $\Ior$. As a consequence, for any $I\in\Ior$, $\Mwf_I\subseteq\Mwf$. Details of this new definition are provided in~\autoref{sec:s2}. 

As in the case of the Yorioka ideals, we greatly expect that the directed scheme $\vec\Mwf$ is useful to understand the combinatorics of $\MAwf$ and its cardinal characteristics. Initially, we intended to approach the following questions stated in~\cite{CMR2}.

\begin{question}
\begin{multicols}{2}
\begin{enumerate}[label=\rm(\arabic*)]
    \item Is $\add(\MAwf)\leq\bfrak$?   
    \item Is $\add(\MAwf)=\non(\MAwf)$?
\end{enumerate}
\end{multicols} 
\end{question}

Both questions cannot have positive answers simultaneously because $\add(\MAwf)\leq\bfrak$ implies $\add(\MAwf)=\add(\Mwf)$ (considering that $\add(\Mwf)=\min\{\bfrak,\non(\MAwf)\}$, due to Pawlikowski~\cite{paw85}), while $\bfrak<\non(\MAwf)$ is consistent with $\thzfc$ (\cite[Thm.~2.4]{paw85}).

Instead of solving the previous questions, the directed scheme $\vec\Mwf$ gives us interesting characterizations of category.

\begin{teorema}[\autoref{thm:b1}]\label{thm:a4}
For all $I\in\Ior$, we have 
\begin{align*}
    \add(\Mwf_I) =\add(\Mwf) & =\min\{\bfrak,\minnon(\vec\Mwf)\},\\\cof(\Mwf_I)=\cof(\Mwf) & =\max\{\dfrak,\supcov(\vec\Mwf)\}.
\end{align*}
\end{teorema}
The characterization of $\add(\Mwf)$ in this theorem is already known from Pawlikowski’s $\add(\Mwf)=\min\{\bfrak,\non(\MAwf)\}$ (\cite{paw85}, and by \autoref{thm:a10}, $\minnon(\vec\Mwf) = \non(\MAwf)$).

Consequently, it is clear that $\minadd(\vec \Mwf) = \supadd(\vec\Mwf) = \add(\Mwf)$ and $\mincof(\vec \Mwf) =\supcof(\vec\Mwf) = \cof(\Mwf)$. 

Thanks to the directed scheme of ideals $\vec\Mwf$, we can estimate an upper bound of the cofinality of $\MAwf$. Combining \autoref{thm:a10} and \ref{thm:a4} results in:

\begin{corollary}\label{thm:a3}
For any $I\in\Ior$, 
\[\add(\Mwf)\leq\add(\MAwf)\lay\cof(\MAwf)\leq\cov(([\mathrm{cof}(\Mwf)]^{<\scriptstyle{\mathrm{add}(\Mwf)}})^{\dfrak})\]    
\end{corollary}

Note that the first inequality in the prior theorem is already known (\autoref{lem:a0}~\ref{lem:a0:1}). 

Below, we illustrate the connections between the cardinal characteristics associated with the ideals $\Mwf_I$ and some other classical characteristics (\autoref{cichonext}).

\begin{teorema}\label{thm:a6}
For $I\in\Ior$: 
\begin{enumerate}[label=\rm(\arabic*)]
    \item If $\sum_{k<\omega}2^{-|I_k|}<\infty$, then $\non(\Mwf_I)\leq\non(\Ewf)$ and $\cov(\Ewf)\leq\cov(\Mwf_I)$.
    \item If $\sum_{k<\omega}2^{-|I_k|}=\infty$, then $\cov(\Nwf)\leq\non(\Mwf_I)$ and $\cov(\Mwf_I)\leq\non(\Nwf)$.
\end{enumerate}
\end{teorema}

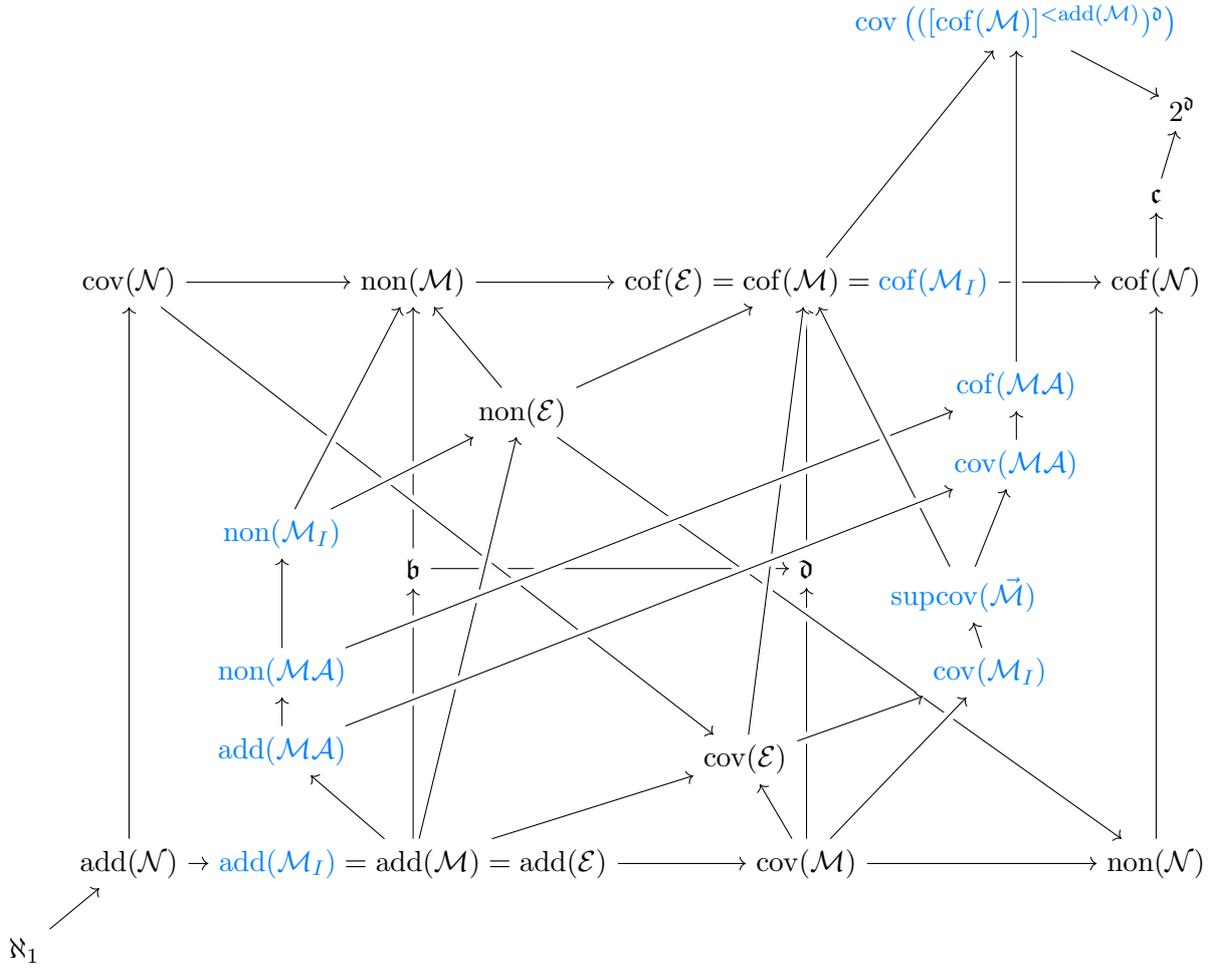
\begin{figure}[ht]
\centering
\begin{tikzpicture}[scale=1.15]
\small{
\node (aleph1) at (-1.45,1.3) {$\aleph_1$};
\node (addn) at (-0.25,2.3){$\add(\Nwf)$};
\node (covn) at (-0.25,9){$\cov(\Nwf)$};
\node (nonn) at (11.5,2.3) {$\non(\Nwf)$};
\node (cfn) at (11.5,9) {$\cof(\Nwf)$};
\node (addm) at (3,2.3) {$\dodger{\add(\Mwf_I)}=\add(\Mwf)=\add(\Ewf)$};
\node (covm) at (7.5,2.3) {$\cov(\Mwf)$};
\node (nonm) at (3,9) {$\non(\Mwf)$};
\node (cfm) at (7.5,9) {$\cof(\Ewf)=\cof(\Mwf)=\dodger{\cof(\Mwf_I)}$};
\node (b) at (3,5.7) {$\bfrak$};
\node (d) at (7.5,5.7) {$\dfrak$};
\node (cove) at (6.8,3.5) {$\cov(\Ewf)$};
\node (none) at (4.25,7.5) {$\non(\Ewf)$};
\node (c) at (11.5,10) {$\cfrak$};

\node (addma) at (1.5,3.6) {\dodger{$\add(\MAwf)$}};
\node (nonma) at (1.5,4.5) {\dodger{$\non(\MAwf)$}};
\node (nonmI) at (1.5, 6.1){\dodger{$\non(\Mwf_I)$}};
\node (covmI) at (9.6,4.5) {\dodger{$\cov(\Mwf_{I})$}};
\node (supmI) at (9.3, 5.4){\dodger{$\supcov(\vec\Mwf)$}};
\node (covma) at (9.9,6.9) {\dodger{$\cov(\MAwf)$}};
\node (cofma) at (9.9,7.8) {\dodger{$\cof(\MAwf)$}};
\node (uppercofma) at (9.9,12) {\dodger{$\cov\left(([\cof(\Mwf)]^{<\add(\Mwf)})^\dfrak\right)$}};
\node (2d) at (11.8,11) {$2^\dfrak$};

 \draw (aleph1) edge[->] (addn);
  \draw(addn) edge[->] (covn)
      (covn) edge [->] (nonm)
      (covm) edge [->] (nonn)
      (nonm)edge [->] (cfm)
      (cfm)edge [->] (cfn)
      (cfn) edge[->] (c);

\draw
   (addn) edge [->]  (addm)
   (addm) edge [->]  (covm)
   (nonn) edge [->]  (cfn);
\draw (addm) edge [->] (b)
      (b)  edge [->] (nonm);
\draw (covm) edge [->] (d)
      (d)  edge[->] (cfm);
\draw (b) edge [->] (d);

\draw (none) edge[line width=.15cm,white,-]  (nonn);
\draw (none) edge [->]  (nonn);
\draw (none) edge [->]  (nonm);
\draw (none) edge [->]  (cfm);
\draw (cove) edge [<-]  (covm);
\draw (cove) edge [<-]  (addm);

\draw  (cove) edge[line width=.15cm,white,-] (covn);
\draw  (cove) edge [<-]  (covn);

\draw  (cove) edge[line width=.15cm,white,-] (covmI);
\draw  (cove) edge [->]  (covmI);

\draw  (cove) edge[line width=.15cm,white,-] (cfm);
\draw  (cove) edge [->]  (cfm);

\draw  (addm) edge[line width=.15cm,white,-] (none);
\draw  (addm) edge [->]  (none);


\draw (nonmI) edge[line width=.15cm,white,-]  (none);
\draw (nonmI) edge [->]  (none);

\draw (nonmI) edge[line width=.15cm,white,-]  (nonm);
\draw (nonmI) edge [->]  (nonm);

\draw (nonma) edge [->]  (nonmI);
\draw (addma) edge [->]  (nonma);
\draw (addm) edge [->]  (addma);

\draw(covm) edge[line width=.15cm,white,-]  (covmI) ;
\draw(covm) edge [->]  (covmI) ;

\draw (supmI) edge [<-]  (covmI);
\draw (supmI) edge [->]  (covma);
\draw (covma) edge [->]  (cofma);

\draw (cofma) edge[line width=.15cm,white,-]  (uppercofma);
\draw (cofma) edge [->]  (uppercofma);

\draw (nonma) edge[line width=.15cm,white,-]  (cofma);
\draw (nonma) edge [->]  (cofma);

\draw (supmI) edge[line width=.15cm,white,-]  (cfm);
\draw (supmI) edge [->]  (cfm);

\draw (addma) edge[line width=.15cm,white,-]  (covma);
\draw (addma) edge [->]  (covma);

\draw (uppercofma) edge [->]  (2d);
\draw (c) edge [->]  (2d);
\draw (cfm) edge [->]  (uppercofma);

}
\end{tikzpicture}
\caption{Cicho\'n's diagram including the cardinal characteristics associated with our ideals, and $\add(\Ewf) = \add(\Mwf)$ and $\cof(\Ewf) = \cof(\Mwf)$ due to Bartoszy{\'n}ski and Shelah~\cite{BS1992}. 
The inequality $\non(\Mwf_I)\leq\non(\Ewf)$ and $\cov(\Ewf)\leq\cov(\Mwf_I)$ holds whenever $\sum_{k<\omega}2^{-|I_k|}<\infty$.
}
\label{cichonext}
\end{figure}

\autoref{thm:a4} and~\ref{thm:a6} will be proved in \autoref{sec:s2}. 
At the end of this section, we further show the connections between $\Mwf_I$ and $\Rbf_b$ ($b\in \baire$), a relational system that helps us in~\cite{CMR2} dissect and reformulate Barotszy\'nski's and Judah's~\cite{bartJudah} characterization of $\non(\MAwf)$. 

In \autoref{sec:s4} we present consistency results. The authors proved in~\cite{CMR2} the consistency of $\cov(\MAwf)<\cfrak$, even $\cov(\MAwf)<\non(\SNwf)$. Since $\MAwf\subseteq\SNwf$, $\cov(\SNwf)\leq\cov(\MAwf)$, but $\cov(\SNwf)=\cfrak$ holds in Sacks model (see e.g.~\cite{CMRM}). We asked about the consistency of $\cov(\NAwf)<\cfrak$ in~\cite{CMR2}, which we solve in this paper. 

\begin{teorema}\label{thm:a11}
    It is consistent with $\thzfc$ that $\cov(\NAwf)<\cfrak$.
\end{teorema}

To prove this, we refine Shelah's characterization of $\NAwf$ from \autoref{thm:a0}~\ref{thm:a0:NA} and show a relevant Tukey connection that gives us a suitable upper bound of $\cov(\NAwf)$. 

A final consistency result is the following.
\begin{teorema}\label{thm:a99}
    It is consistent with $\thzfc$ that $\cof(\MAwf)<\cfrak$, even $\cof(\MAwf)<\non(\SNwf)$.
\end{teorema}

The directed scheme $\vec\Mwf$ helps us see that this inequality holds in our model of $\cov(\MAwf)<\non(\SNwf)$ presented in~\cite{CMR2}.

\section{Directed schemes of ideals}\label{sec:s1}

Most of the results of this section (let alone this paper) are presented in terms of \emph{relational systems}. Our review of relational systems is based on~\cite{Vojtas,blass}. 
We say that $\Rbf=\la X, Y, \sqsubset\ra$ is a \textit{relational system} if it consists of two non-empty sets $X$ and $Y$ and a relation $\sqsubset$.
\begin{enumerate}[label=(\arabic*)]
    \item A set $F\subseteq X$ is \emph{$\Rbf$-bounded} if $\exists\, y\in Y\ \forall\, x\in F\colon x \sqsubset y$. 
    \item A set $E\subseteq Y$ is \emph{$\Rbf$-dominating} if $\forall\, x\in X\ \exists\, y\in E\colon x \sqsubset y$. 
\end{enumerate}
We associate two cardinal characteristics with this relational system:
\begin{itemize}
    \item[{}] $\bfrak(\Rbf):=\min\{|F|\colon  F\subseteq X  \text{ is }\Rbf\text{-unbounded}\}$ the \emph{bounding number of $\Rbf$}, and
    
    \item[{}] $\dfrak(\Rbf):=\min\{|D|\colon  D\subseteq Y \text{ is } \Rbf\text{-dominating}\}$ the \emph{dominating number of $\Rbf$}.
\end{itemize}
The \emph{dual of $\Rbf$} is the relational system $\Rbf^\perp := \la Y,X,\sqsubset^\perp\ra$ where $y \sqsubset^\perp x$ means $x \not\sqsubset y$. Note that $\bfrak(\Rbf^\perp) = \dfrak(\Rbf)$ and $\dfrak(\Rbf^\perp) = \bfrak(\Rbf)$.

For instance, a preorder $\la P,\leq\ra$ (i.e.\ $\leq$ is reflexive and transitive in $P$) can be viewed as the relational system $\la P,P,\leq\ra$, simply denoted by $P$ when $\leq$ is understood. In this context, $\dfrak(P)$ is usually denoted by $\cof(P)$ and $\cf(P)$, and called the \emph{cofinality of $P$}.



The following are representative directed preorders, mostly present in this work.

\begin{definition}\label{ex:c0}
Define the relation $\leq^*$ on $\baire$ by $x\leq^* y$ means $\forall^\infty n<\omega\colon x(n)\leq y(n)$. 
It is clear that $\la\baire\leq^*\ra$ is a directed preorder and 
we use $\baire$ to denote its relational system. Then $\bfrak:=\bfrak(\baire)$ and $\dfrak:=\dfrak(\baire)$ are the well-known \emph{bounding number} and \emph{dominating number}, respectively.
\end{definition}


\begin{example}\label{exa:c2}
Let $\Iwf$ be an ideal on $X$ containing $[X]^{<\aleph_0}$. As a relational system, $\Iwf$ is the partial order $\la\Iwf,\subseteq\ra$, which is directed. Another relational system associated with $\Iwf$ is $\Cv_\Iwf:=\la X,\Iwf,\in\ra$. The cardinal characteristics associated with $\Iwf$ are obtained by these relational systems:
\begin{align*}
    \bfrak(\Iwf) & =\add(\Iwf), &
    \dfrak(\Iwf) & =\cof(\Iwf),\\
    \bfrak(\Cv_\Iwf) & =\non(\Iwf), &
    \dfrak(\Cv_\Iwf) & =\cov(\Iwf).
\end{align*}
\end{example}

 For relational systems $\Rbf=\la X,Y,\sqsubset\ra$ and $\Rbf'=\la X',Y',\sqsubset'\ra$,  a pair $(\Psi_-,\Psi_+)$ is a \emph{Tukey connection from $\Rbf$ into $\Rbf'$} if 
 $\Psi_-\colon X\to X'$ and $\Psi_+\colon Y'\to Y$ are functions satisfying  
 \[\forall\, x\in X\ \forall\, y'\in Y'\colon \Psi_-(x) \sqsubset' y' \Rightarrow x \sqsubset \Psi_+(y').\] 
 We say that $\Rbf$ is \emph{Tukey below} $\Rbf'$, denoted by $\Rbf\leqT\Rbf'$, if there is a Tukey connection from $\Rbf$ into $\Rbf'$, and we 
 say that $\Rbf$ and $\Rbf'$ are \emph{Tukey equivalent}, denoted by $\Rbf\eqT\Rbf'$, if $\Rbf\leqT\Rbf'$ and $\Rbf'\leqT\Rbf$. It is well-known that $\Rbf\leqT\Rbf'$ implies $(\Rbf')^\perp \leqT \Rbf^\perp$, $\bfrak(\Rbf')\leq\bfrak(\Rbf)$ and $\dfrak(\Rbf)\leq\dfrak(\Rbf')$. Hence, $\Rbf\eqT\Rbf'$ implies $\bfrak(\Rbf')=\bfrak(\Rbf)$ and $\dfrak(\Rbf)=\dfrak(\Rbf')$.

\begin{example}\label{ex:c4}
\ 
\begin{enumerate}[label=\rm (\arabic*)]
    \item \label{ex:c4:0} Let $\Iwf$ and $\Jwf$ be ideals on $X$. If $\Iwf\subseteq\Jwf$ then $\Cv_\Jwf\leqT\Cv_\Iwf$.
    
    \item\label{ex:c4:1} For any ideal $\Iwf$ on $X$, $\Cv_\Iwf\leqT\Iwf$ and $\Cv^\perp_\Iwf\leqT\Iwf$. These determine some of the inequalities in \autoref{diag:idealI}.
    
    \item\label{ex:c4:2} If $\theta'\leq\theta$ are cardinals and $\theta\leq|X|\leq|X'|$, then $\Cv_{[X]^{<\theta}}\leqT\Cv_{[X']^{<\theta'}}$ and $[X]^{<\theta}\leqT[X']^{<\theta}$.
    
    \item\label{ex:c4:3} For any cardinal $\mu$, $\Cv_{[\mu]^{<\mu}}\leqT\mu \leqT\Cv^\perp_{[\mu]^{<\mu}} \leqT[\mu]^{<\mu}$. In the case when $\mu$ is regular, $[\mu]^{<\mu}\leqT \Cv_{[\mu]^{<\mu}}$, so $\add([\mu]^{<\mu})=\cof([\mu]^{<\mu})=\mu$.
    
    \item\label{ex:c4:4} Whenever $S$ is a directed preorder, $S\leqT\Cv_{[\dfrak(S)]^{<\bfrak(S)}}$.
    
\end{enumerate}
\end{example}

\begin{example}[{\cite{blass}}]\label{ex:c1}
    The directed preorder $\Ior=\la\Ior,\sqsubseteq\ra$ is Tukey-equivalent with $\baire$. Therefore, $\bfrak(\Ior)=\bfrak$ and $\dfrak(\Ior)=\dfrak$.
\end{example}

To prove the results of this work, we employ the following operations of relational systems.

\begin{definition}\label{def:prodrel}
Let $\vec{\Rbf} = \seq{\Rbf_i}{i\in K}$, $\Rbf_i = \la X_i,Y_i,\sqsubset_i\ra$, be a sequence of relational systems. Define the following relational systems.
\begin{enumerate}[label= \normalfont (\arabic*)]
    \item $\bigvee\vec\Rbf = \bigvee_{i\in K}\Rbf_i :=\left\la \bigcup_{i\in K}\{i\}\times X_i, \prod_{i\in K} Y_i, \sqsubset^\vee  \right\ra$ such that $(i,x)\sqsubset^\vee y$ iff $x\sqsubset_i y_i$. 
    \item $\prod \vec{\Rbf} = \prod_{i\in K} \Rbf_i:=\left\la \prod_{i\in K}X_i, \prod_{i\in K}Y_i, \sqsubset^\otimes \right\ra$ where $x \sqsubset^\otimes y$ iff $x_i\sqsubset_i y_i$ for all $i\in K$.

     When $\Rbf_i = \Rbf$ for all $i\in K$, we write $\Rbf^K := \prod \vec{\Rbf}$.
    \item When $K=\gamma$ is an ordinal, the \emph{sequential composition} is
    \[\vec\Rbf^; := \left\la \prod_{\alpha<\gamma}{}^{\prod_{i<\alpha}Y_i} X_\alpha, \prod_{\alpha<\gamma} Y_\alpha, \sqsubset^; \right\ra\]
    where $\bar f = \seq{f_\alpha}{\alpha<\gamma} \sqsubset^; y$ iff $f_\alpha(y\frestr \alpha)\sqsubset_\alpha y_\alpha$ for all $\alpha<\gamma$. 
\end{enumerate}
When $K=\{0,1\}$, we denote the previous relational systems by $\Rbf_0\vee \Rbf_1$, $\Rbf_0\times \Rbf_1$ and $(\Rbf_0;\Rbf_1)$. 
\end{definition}

Blass~\cite[Thm.~4.11]{blass} states a result about the cardinal characteristics associated with the operations above for $K=\{0,1\}$, which we generalize below.

\begin{theorem}\label{products}
    Let $\vec{\Rbf} = \seq{\Rbf_i}{i\in K}$ be a sequence of relational systems $\Rbf_i = \la X_i,Y_i,\sqsubset_i\ra$.
    \begin{enumerate}[label = \normalfont (\alph*)]
        \item\label{a251a} $\bigvee \vec\Rbf$ is the supremum of $\set{\Rbf_i}{i\in K}$ with respect to $\leqT$, i.e.\ $\Rbf_i\leqT \bigvee \vec\Rbf$ for all $i\in K$ and, whenever $\Rbf'$ is another relational system such that $\Rbf_i\leqT \Rbf'$ for all $i\in K$, $\bigvee \vec\Rbf \leqT \Rbf'$.

        \item\label{a251b} $\bigvee \vec\Rbf \leqT \prod \vec\Rbf$ and $\prod \vec\Rbf \leqT \vec\Rbf^;$, the latter when $K$ is an ordinal.

        \item\label{a251c} $\bfrak(\bigvee \vec\Rbf) = \min_{i\in K}\bfrak(\Rbf_i)$ and $\dfrak(\bigvee \vec\Rbf) = \sup_{i\in K}\dfrak(\Rbf_i)$.

        \item\label{a251d} $\bfrak(\prod \vec\Rbf) = \min_{i\in K}\bfrak(\Rbf_i)$ and $\sup_{i\in K}\dfrak(\Rbf_i) \leq \dfrak(\prod \vec\Rbf) \leq \prod_{i\in K}\dfrak(\Rbf_i)$. \cite[Fact~3.9]{CarMej23}

        \item\label{a251e} When $K$ is an ordinal, $\bfrak(\vec\Rbf^;) = \min_{i\in K}\bfrak(\Rbf_i)$ and $\dfrak(\vec\Rbf^;)\leq \prod_{i\in K}\dfrak(\Rbf_i)$. Equality holds when $K$ is finite. 

    \end{enumerate}
\end{theorem}
\begin{proof}
    Without loss of generality, we can assume that $K = \gamma$ is an ordinal.

    For $i\in K$, $\Rbf_i\leqT \bigvee \vec\Rbf$ is obtained by the maps $x\mapsto (i,x)$ and $y\mapsto y(i)$.

    If, for each $i\in K$, $(\Psi^i_-,\Psi^i_+)$ is a Tukey connection for $\Rbf_i\leqT \Rbf'$, then $\bigvee \vec\Rbf \leqT \Rbf'$ is obtained by the maps $(i,x)\mapsto \Psi^i_-(x)$ and $y'\mapsto \seq{\Psi^i_+(y')}{i\in K}$. This shows~\ref{a251a}.

    It is easy to show that $\Rbf_i\leqT \prod \vec\Rbf$ for all $i\in K$, so $\bigvee \vec\Rbf\leqT \prod \vec\Rbf$ by~\ref{a251a}. 
    On the other hand, $\prod \vec\Rbf \leqT \vec\Rbf^;$ via the maps $x \mapsto \seq{f^x_\alpha}{\alpha<\gamma}$ such that each $f^x_\alpha$ is the constant map with value $x_\alpha$, and $y\mapsto y$.
    This shows~\ref{a251b}.

    By~\ref{a251a} it is clear that $\bfrak(\bigvee \vec\Rbf) \leq\bfrak(\Rbf_i)$ and $\dfrak(\Rbf_i) \leq \dfrak(\bigvee \vec\Rbf)$ for all $i\in K$, so $\bfrak(\bigvee \vec\Rbf) \leq\min_{i\in K}\bfrak(\Rbf_i)$ and $\sup_{i\in K}\dfrak(\Rbf_i) \leq \dfrak(\bigvee \vec\Rbf)$.


    For each $i\in K$, pick some $\Rbf_i$-dominating $D^i:=\set{y^i_\xi}{\xi<\dfrak(\Rbf_i)}\subseteq Y_i$. When $\dfrak(\Rbf_i)\leq \xi<\sup_{i\in K}\dfrak(\Rbf_i)$, set $y^i_\xi:= y^i_0$. So define $y_\xi:= \seq{y^i_\xi}{i\in K}$ for $\xi<\sup_{i\in K}\dfrak(\Rbf_i)$. Then the set $\set{y_\xi}{\xi<\sup_{i\in K}\dfrak(\Rbf_i)}$ is $\bigvee \vec\Rbf$-dominating, so $\dfrak(\bigvee \vec\Rbf)\leq \sup_{i\in K}\dfrak(\Rbf_i)$. On the other hand, it is easy to see that $\prod_{i\in K}D_i$ is dominating for $\prod \vec\Rbf$ and $\vec\Rbf^;$, so $\dfrak(\prod \vec\Rbf) \leq \dfrak(\vec\Rbf^;)\leq \prod_{i\in K}\dfrak(\Rbf_i)$.


    By~\ref{a251b} and what we have proved so far, $\bfrak(\vec\Rbf^;)\leq \bfrak(\prod \vec\Rbf) \leq \bfrak(\bigvee \vec\Rbf) \leq \min_{i\in K}\bfrak(\Rbf_i)$ and $\sup_{i\in K}\dfrak(\Rbf_i) = \dfrak(\bigvee \vec\Rbf) \leq \dfrak(\prod \vec\Rbf) \leq \dfrak(\vec\Rbf^;)\leq \prod_{i\in K}\dfrak(\Rbf_i)$. So, to conclude~\ref{a251c}--\ref{a251d} and part of~\ref{a251e}, it is enough to show that $\min_{i\in K}\bfrak(\Rbf_i)\leq \bfrak(\vec\Rbf^;)$. 
    Let $F\subseteq \prod_{\alpha<\gamma}{}^{\prod_{i<\alpha}Y_i}X_\alpha$ be a set of size ${<}\min_{\alpha<\gamma}\bfrak(\Rbf_\alpha)$. Find $y_\alpha\in Y_\alpha$ by recursion on $\alpha<\gamma$ such that $f_\alpha(y\frestr\alpha)\sqsubset_\alpha y_\alpha$ for all $\bar f = \seq{f_\alpha}{\alpha<\gamma} \in F$. This is fine because $|F|< \min_{\alpha<\gamma}\bfrak(\Rbf_\alpha)$. Therefore, $F$ is $\vec\Rbf^;$-bounded by $y:=\seq{y_\alpha}{\alpha<\gamma}$. 

    It remains to show the equality of the dominating numbers in~\ref{a251e} when $K$ is finite. Since sequential composition is associative (modulo Tukey-equivalence), it is enough to assume that $K=\{0,1\}$. So let $D\subseteq Y_0\times Y_1$ of size ${<}\dfrak(\Rbf_0)\cdot\dfrak(\Rbf_1)$. Set $D_0$ as the set of $y_0\in Y_0$ such that $D_1^{y_0}:=\set{y_1}{(y_0,y_1)\in D}$ is $\Rbf_1$-dominating. Then $D_0$ is not $\Rbf_0$-dominating, otherwise $|D_0|\geq \dfrak(\Rbf_0)$ and, since $|D^{y_0}_1|\geq \dfrak(\Rbf_1)$ for $y_0\in D_0$, we would get that $|D|\geq \dfrak(\Rbf_0)\cdot\dfrak(\Rbf_1)$, a contradiction.

    Hence, there is some $x_0\in X_0$ such that $x_0\not\sqsubset_0 y_0$ for all $y_0\in D_0$. Define $f_0\colon\{\emptyset\}\to X_0$ such that $f_0(\emptyset):=x_0$ and let $f_1\colon Y_0\to X_1$ be any map such that, whenever $y_0\notin D_0$,  $f_1(y_0)\not\sqsubset_1 y_1$ for any $y_1\in D^{y_0}_1$. We then get that $\bar f:=(f_0,f_1) \not\sqsubset^; (y_0,y_1)$ for all $(y_0,y_1)\in D$. 
\end{proof}


\begin{remark}
    The equality of the $\dfrak$-number in \autoref{products}~\ref{a251e} does not hold in general when $K$ is infinite. Consider the case when $K=\omega$ and $\Rbf_i = \la \omega,\neq\ra$ for $i<\omega$. Then $\vec\Rbf^;\eqT \prod \vec\Rbf$ and $\dfrak(\prod \vec\Rbf) = \cov(\Mwf)$ (this is a consequence of the well-known Bartoszy\'nski's and Miller's characterization of $\cov(\Mwf)$, see e.g.~\cite[Sec.~5]{CM23}), but $\prod_{i<\omega}\dfrak(\Rbf_i) = \cfrak$. 

    To see this Tukey equivalence, notice that $\prod \vec\Rbf = \la \baire, \neq^\bullet\ra$ where $x\neq^\bullet y$ means that $x(n)\neq y(n)$ for all $n<\omega$, while $\vec\Rbf^; =\la {}^{{}^{<\omega}\omega} \omega, \baire, \neq^\bullet\ra $ where we interpret $\varphi \neq^\bullet y$ as $\varphi(y\frestr n) \neq y(n)$ for all $n<\omega$. It is clear that $\prod \vec\Rbf \eqT \la {}^{{}^{<\omega}\omega} \omega, \neq^\bullet\ra$ where $\varphi \neq^\bullet \psi$ is interpreted as $\varphi(s) \neq \psi(s)$ for all $s\in{}^{<\omega}\omega$. 

    It is enough to show that $\vec\Rbf^; \leqT \la {}^{{}^{<\omega}\omega} \omega, \neq^\bullet\ra$. Let $F$ be the identity function on ${}^{{}^{<\omega}\omega}\omega$ and define $F'\colon {}^{{}^{<\omega}\omega} \omega \to \baire$ such that $F'(\psi):= y_\psi$ is defined recursively by $y_\psi(n):= \psi(y_\psi\frestr n)$. It is clear that $(F,F')$ is the required Tukey connection. 

    Another example is $\Rbf_i:=\la \omega,\leq \ra$ for $i<\omega$, where $x\leq y$ means $x(n)\leq y(n)$ for all $n<\omega$, where we also get $\vec \Rbf^; \eqT \prod \vec\Rbf$, whose $\dfrak$-number is $\dfrak$.
\end{remark}

The following result is a useful trick to produce Tukey connections.

\begin{lemma}[cf.~{\cite[4.1.6]{Vojtas}}]\label{a255}
    For $i\in \{0,1\}$, let $\Rbf_i=\la X_i,Y_i,\sqsubset_i\ra$ be a relational system, and let $\Rbf=\la S,Y,\sqsubset\ra$ be a relational system such that $S$ is directed and satisfying that, for $s,s'\in S$ and $y\in Y$, if $s'\sqsubset y$ and $s\leq_S s'$ then $s\sqsubset y$.

    If $\Rbf_0\leqT \Rbf$ and $\Rbf_1\leqT \Rbf$, then $\Rbf_0\times \Rbf_1\leqT \Rbf$.
\end{lemma}
\begin{proof}
    For $i\in\{0,1\}$, consider the Tukey connections $\Psi^i_-\colon X_i\to S$ and $\Psi^i_+\colon Y\to Y_i$ for $\Rbf_i\leqT \Rbf$. Define $\Psi_-\colon X_0\times X_1\to S$ such that $\Psi_-(x_0,x_1)$ is above $\Psi^0_-(x_0)$ and $\Psi^1_-(x_1)$ in $S$, and define $\Psi_+\colon Y\to Y_0\times Y_1$ by $\Psi_+(y):=(\Psi^0_+(y_0),\Psi^1_+(y_1))$. 
\end{proof}

Powers of ideals can be introduced using the product of relational systems in the following ways.

\begin{definition}[{\cite[Def.~3.14]{CM23}}]\label{def:a16}
Given an ideal $\Iwf$ on $X$ and a set $w$, define $\Iwf^{(w)}$ as the ideal on ${}^w X$ generated by the sets of the form $\prod_{i\in w}A_i$ with $\seq{A_i}{i\in w}\in{}^w \Iwf$. Denote $\add(\Iwf^w):=\bfrak(\Iwf^{(w)})$, $\cof(\Iwf^w):=\dfrak(\Iwf^{(w)})$, $\non(\Iwf^w):=\bfrak(\Cv_{\Iwf^{(w)}})$ and $\cov(\Iwf^w):=\dfrak(\Cv_{\Iwf^{(w)}})$.
\end{definition}

\begin{fact}[{\cite[Fact~3.15]{CM23}}]\label{fct:a2}
Let $w$ be a set and let $\Iwf$ be an ideal on $X$. Then:
\begin{enumerate}[label=\rm (\alph*)]
    \item $\Iwf^w\eqT\Iwf^{(w)}$. 
    
    \item $\Cv_\Iwf^w\eqT\Cv_{\Iwf^{(w)}}$.
    
    \item $\add(\Iwf^w)=\add(\Iwf)$ and $\non(\Iwf^w)=\non(\Iwf)$.
    
    \item $\cov(\Iwf)\leq\cov(\Iwf^w)\leq \cov(\Iwf)^{|w|}$ and $\cof(\Iwf)\leq\cof(\Iwf^w)\leq \cof(\Iwf)^{|w|}$.
\end{enumerate}
\end{fact}

We can prove \autoref{thm:a10} in terms of relational systems and Tukey connections. The following result is the starting point to deal with the covering and uniformity numbers. 

\begin{theorem}\label{lem:c0}
Let $X$ be a set, $S$ a directed preorder and $D\subseteq S$ cofinal in $S$.  Assume that $\vec\Iwf = \seq{\Iwf_s}{s\in S}$ is a directed scheme of ideals on $X$. Then
\begin{enumerate}[label=\rm(\arabic*)]
    \item\label{lem:c0:1} If $s\leq s'$, then $\Cv_{\Iwf_s}\leqT\Cv_{\Iwf_{s'}}\leqT \Cv_{\Iwf^*}$. 
    \item\label{lem:c0:1.5} $\Cv_{\Iwf^*}\leqT \prod_{s\in D}\Cv_{\Iwf_s}$.
    \item\label{lem:c0:2} $\supcov(\vec{\Iwf})\leq\cov(\Iwf^*) \leq \dfrak\left(\prod_{s\in D}\Cv_{\Iwf_s}\right)$ and $\non(\Iwf^*)=\minnon(\vec{\Iwf})$
\end{enumerate}
\end{theorem}
\begin{proof}
\ref{lem:c0:1} follows directly by \autoref{ex:c4}~\ref{ex:c4:0}, and~\ref{lem:c0:2} is a consequence of~\ref{lem:c0:1},~\ref{lem:c0:1.5} and \autoref{products}~\ref{a251d}. 
So it remains to show~\ref{lem:c0:1.5}. Define the functions $\Psi_-:X\to\prod_{s\in D}X$ and $\Psi_+:\prod_{s\in D}\Iwf_s\to\Iwf^*$ where $\Psi_-(x)$ is the constant sequence with value $x$, and $\Psi_+(\seq{A_s}{ s\in D}):=\bigcap_{s\in D}A_s$, which is in $\Iwf^*$ because $D$ is cofinal in $S$. These maps form the desired Tukey connection.
\end{proof}

We now look at the additivity and cofinality numbers. 

\begin{theorem}\label{thm:c1}
Let $\vec\Iwf = \seq{\Iwf_s}{s\in S}$ be a directed scheme of ideals on $X$.  Then, for any cofinal $D\subseteq S$, 
\[\Iwf^*\leqT\prod_{s\in D}\Iwf_s\leqT\prod_{s\in D}\Cv_{[\cof(\Iwf_s)]^{<\add(\Iwf_s)}}\leqT\Cv^{D}_{[\cof(\vec\Iwf)]^{<\minadd(\vec\Iwf)}}.\] 
As a consequence, 
\[\minadd(\vec{\Iwf})\leq\add(\Iwf^*)\] 
and, whenever $D$ is a witness of $\cof(S)$, 
\[\cof(\Iwf^*)\leq\dfrak\left(\prod_{s\in D}\Iwf_s\right)\leq\cov\left(\left([\supcof(\vec\Iwf)]^{<\minadd(\vec\Iwf)}\right)^{\cof(S)}\right).\]
\end{theorem}
\begin{proof}
The inequality $\Iwf^*\leqT\prod_{s\in D}\Iwf_s$ is easy to show: define $\Psi_-:\Iwf^*\to\prod_{s\in D}\Iwf_s$ such that $\Psi_-(A)$ is the constant sequence with value $A$, and define $\Psi_+:\prod_{s\in D}\Iwf_s\to\Iwf^*$ such that $\Psi_+(\seq{A_s}{s\in D}):=\bigcap_{s\in D}A_s$, which is in $\Iwf^*$ because $D$ is cofinal in $S$. Note that $(\Psi_-,\Psi_+)$ is the required Tukey connection.

On the other hand, using ~\ref{ex:c4:2} and~\ref{ex:c4:4} of~\autoref{ex:c4}, we obtain 
\[\prod_{s\in D}\Iwf_s\leqT\prod_{s\in D}\Cv_{[\cof(\Iwf_s)]^{<\add(\Iwf_s)}}\leqT\Cv^{D}_{[\supcof(\vec\Iwf)]^{<\minadd(\vec\Iwf)}}.\qedhere\]
\end{proof}

\autoref{thm:a10} is a direct consequence of \autoref{lem:c0} and \autoref{thm:c1}.

\section{Directed scheme for the meager additive ideal}\label{sec:s2}

The purpose of this section is to introduce the directed scheme of ideals $\seq{\Mwf_I}{I\in\Ior}$, which reformulates $\MAwf = \bigcap\set{\Mwf_I}{I\in\Ior}$, and study these ideals. In particular, we prove~\autoref{thm:a4} and \ref{thm:a6}.

We start with reviewing the following cofinal family of the meager ideal. 

\begin{definition}\label{def:b0}
Let $I\in\Ior$ and $y\in\cantor$. Define 
\[B_{y,I}:=\set{x\in\cantor}{\forall^\infty n\in\omega\colon x{\upharpoonright}I_n\neq y{\upharpoonright}I_n}.\]
\end{definition}

A pair $(y,I)\in 2^\omega\times \Ior$ is known as a \emph{chopped real}, and these are used to produce a cofinal family of meager sets.
It is clear that $B_{y,I}$ is a meager subset 
of $\cantor$ (see e.g.~\cite{blass}).

\begin{theorem}[{Talagrand~\cite{Tal98}}, see e.g.~{\cite[Prop.~13]{BWS}}]\label{thm:b0}
For every meager set $F\subseteq\cantor$ and $I\in\Ior$ there are $y\in\cantor$ and $I'\in\Ior$ such that $F\subseteq B_{y,I'}$ and each $I'_n$ is the union of finitely many $I_k$'s.
\end{theorem}

Inspired by~\eqref{for:ao} let us introduce a new ideal related to $\MAwf$ as follows.

\begin{definition}\label{def:b2}
 Fix $I\in\Ior$. For $J\in\Ior$ and $y\in\cantor$, denote 
\[A_{J,y}^I:=\set{x\in\cantor}{\forall^\infty n<\omega\, \exists k<\omega\colon I_k\subseteq J_n\text{\ and\ }x{\upharpoonright}I_k=y{\upharpoonright}I_k},\]
which is an $F_\sigma$-meager set, and define \[\Mwf_I:=\set{X\subseteq\cantor}{\exists(J,y)\in\Ior\times\cantor\colon X\subseteq A_{J,y}^I}.\]
Also denote $\vec\Mwf = \seq{\Mwf_I}{I\in\Ior}$. 
\end{definition}

By \autoref{thm:a0}~\ref{thm:a0:MA}, it is clear that $\MAwf=\bigcap_{I\in\Ior}\Mwf_I$, but we will offer a proof in \autoref{cor:b4}. To make sense of the previous definition and conclude that $\vec\Mwf$ is a directed scheme of ideals, 
we aim to show that $\Mwf_I$ is a $\sigma$-ideal and $\Mwf_{I'}\subseteq \Mwf_I$ whenever $I\sqsubseteq I'$ in $\Ior$. The former requires some work, for which we introduce the following directed preorder.

\begin{definition}\label{def:b1}
Fix $I\in\Ior$. For $J, J'\in\Ior$ and $y, y'\in\cantor$ define the relation $(J,y)\sqsubseteq^I(J',y')$ iff either $I\not\sqsubseteq J$ or, for all but finitely many $n<\omega$, there is some $\ell<\omega$ satisfying the following conditions:
\begin{enumerate}[label = \normalfont (\roman*)]
    \item $J_\ell$ contains some $I_k$, and
    \item for all $k<\omega$, $I_k\subseteq J_\ell$ implies $I_k\subseteq J'_n$ and $y{\upharpoonright}I_k=y'{\upharpoonright}I_k$.
\end{enumerate}
\end{definition}

Notice that $I\sqsubseteq J$ and $(J,y) \sqsubseteq^I (J', y')$ implies $I\sqsubseteq J'$.

\begin{lemma}\label{lem:b1.5}
$\la\Ior\times\cantor,\sqsubset^I\ra$  is a directed preorder and $\bfrak(\Ior\times\cantor,\sqsubseteq^I)$ is uncountable.
\end{lemma}
\begin{proof}
 It is easy to show that $\la\Ior\times\cantor,\sqsubseteq^I\ra$ is a preorder, so it is enough to prove that $\bfrak(\Ior\times\cantor,\sqsubseteq^I)$ is uncountable (which implies directedness). 
 Let $(J^m, y^m)\in \Ior\times\cantor$ for $m<\omega$ and, wlog, assume that $I\sqsubseteq J^m$ for all $m<\omega$. Construct a $J=\seq{J_n}{n<\omega}\in\Ior$ such that each $J_n$ contains one interval from $J^m$, say $J^m_{\ell_{n,m}}$, for each $m\leq n$, also demanding that $\set{J^m_{\ell_{n,m}}}{m\leq n}$ is pairwise disjoint and that each member of this set contains an interval from $I$. So we can define a $y\in\cantor$ such that $y{\restriction}J_{\ell_{m,n}}^m=y^m{\restriction}J_{\ell_{m,n}}^m$ for all $m\leq n<\omega$. Then  $(J^m,y^m)\sqsubset^I(J,y)$ for all $m<\omega$.
\end{proof}

\begin{mainlemma}\label{lem:b2}
Let $I, J, J'$ and $y, y'$ as in~\autoref{def:b1}. Then, we have that $A_{J,y}^I\subseteq A_{J',y'}^I$ iff $(J,y)\sqsubset^I(J',y')$.
\end{mainlemma}
\begin{proof}
$``\Leftarrow"$ Assume that $x\in A_{J,y}^{I}$, so $\forall^\infty n<\omega\, \exists k<\omega\colon I_k\subseteq J_n$ and $x{\upharpoonright}I_k=y{\upharpoonright}I_k$. Since $(J,y)\sqsubset^I (J',y')$, for large enough $n$ we can choose an $\ell<\omega$ such that 
\begin{enumerate}[label=\rm($\oplus_\arabic*$)]
    \item \label{lem:b2:a} $\exists k<\omega\colon I_k\subseteq J_\ell$ and $x{\restriction}I_k=y{\restriction}I_k$, and 
    \item \label{lem:b2:b} $\forall k<\omega\colon I_k\subseteq J_\ell\imp I_k\subseteq J'_n$ and $y{\restriction}I_k=y'{\restriction}I_k$.
\end{enumerate}
Observe that \ref{lem:b2:b} implies $I_k\subseteq J'_n$ and $x{\restriction}I_k=y'{\restriction}I_k$ for a $k$ as in~\ref{lem:b2:a}. Hence, $x\in A_{J',y'}^I$.

$``\Rightarrow"$ We prove the contrapositive. Assume $(J,y)\not\sqsubset^I (J',y')$, so $I\sqsubseteq J$, i.e.\ there is some $\ell_0<\omega$ such that $J_\ell$ contains some $I_k$ for all $\ell\geq \ell_0$, and there are infinitely many $n<\omega$ such that:
\begin{enumerate}[label = $(\star)$]
    \item\label{stpr} for all $\ell\geq \ell_0$ there is some $k<\omega$ such that $I_k\subseteq J_\ell$ and either $I_k\not\subseteq J'_n$ or $y{\restriction}I_{k}\neq y'{\restriction}I_k$.
\end{enumerate}
 Let $w$ be an infinite set of $n<\omega$ satisfying~\ref{stpr} and such that $|\set{n\in w}{J_\ell\cap J'_n\neq\emptyset}|\leq1$ for any $\ell<\omega$. Denote by $n_\ell$ the unique member of this set, in case it exists. 

Define $x\in\cantor$ at each $I_k$ as follows: First consider the case when $I_k\subseteq J_\ell$ for some $\ell\geq \ell_0$ (which is unique). In the case $n_\ell$ exists, there is some $k_\ell<\omega$ such that $I_{k_\ell}\subseteq J_\ell$ and either $I_{k_\ell}\not\subseteq J'_{n_\ell}$ or $y{\restriction}I_{k_\ell}\neq y'{\restriction}I_{k_\ell}$. Define $x{\restriction}I_{k}:=y{\restriction}I_{k}$ when $k=k_\ell$, otherwise $x{\restriction}I_{k}$ can be anything different from $y'{\restriction}I_k$; in the case that $n_\ell$ does not exist, set $x{\restriction}I_k:=y{\restriction}I_k$. On the other hand, in the case that $I_k$ is not contained in any $J_\ell$ for $\ell\geq \ell_0$, let $x{\restriction}I_k$ be anything different from $y'{\restriction} I_k$.

Firstly, it is clear that $x\in A_{J, y}^I$. On the other hand, $\forall n\in w\,\forall k<\omega\colon I_k\subseteq J'_n\imp x{\restriction}I_k\neq y'{\restriction}I_k$, i.e, $x\not\in A_{J', y'}^I$. Indeed, for $n\in w$ and $k<\omega$ such that $I_k\subseteq J'_n$, in case $I_k\subseteq J_\ell$ for some $\ell\geq\ell_0$, we have $n=n_\ell$ and, if $k=k_\ell$, then $x{\restriction}I_{k_\ell}=y{\restriction}I_{k_\ell}\neq y'{\restriction}I_{k_\ell}$, else $x{\restriction}I_k\neq y'{\restriction}I_k$ by the definition of $x{\restriction}I_k$; otherwise, $I_k$ is not contained in any $J_\ell$ for $\ell\geq \ell_0$, in which case $x{\restriction}I_k\neq y'{\restriction}I_k$.
\end{proof}

As a direct consequence of \autoref{lem:b2}, not only do we have that any $\Mwf_I$ is a $\sigma$-ideal but also that $\Mwf_I$ is Tukey-equivalent to a clean directed preorder. 

\begin{theorem}\label{cor:b0}
For $I\in\Ior$, $\Mwf_I$ is a $\sigma$-ideal and $\Mwf_I\eqT\la\Ior\times\cantor,\sqsubset^I\ra$.
\end{theorem}

We must show, of course, that $\cantor\notin \Mwf_I$, but this is immediate from $\Mwf_I\subseteq\Mwf$, which we prove in \autoref{lem:b3}.

It is clear that the set $\set{(J,y)\in\Ior\times\cantor}{J\text{ is coarser that }I}$ is cofinal in $\la\Ior\times\cantor,\sqsubset^I\ra$. As a consequence:

\begin{corollary}\label{cor:b1}
We can rewrite 
$$\Mwf_I=\set{x\in\cantor}{\exists(J,y)\in\Ior\times\cantor\colon X\subseteq A_{J,y}^I\text{ and $J$ is coarser than }I}.$$
\end{corollary}

We point out some simple facts below, which let us conclude that $\vec\Mwf$ is a directed scheme of ideals.

\begin{fact}\label{fac:b0}
Let $I, I',J,J'\in\Ior$ and $x,y,y'\in\cantor$. Then:
\begin{enumerate}[label=\rm(\alph*)]
    \item\label{fac:b0:1} $x\in A_{J,y}^I$ iff $(I,x)\sqsubset^I(J,y)$.
    \item\label{fac:b0:2} If $I\sqsubseteq I'$ then $A^{I'}_{J,y} \subseteq A^{I}_{J,y}$. In particular, $\Mwf_{I'}\subseteq\Mwf_I$. 
\end{enumerate}
\end{fact}

We ask:

\begin{question}
 Does $I\sqsubseteq I'$ imply $\Mwf_{I'}\leqT\Mwf_I$?  
\end{question}

Although the answer to the previous question is unknown to us, we have a stronger relation on $\Ior$ from where $\Mwf_{I'}\leqT\Mwf_I$ holds. Namely, for $I,I' \in \Ior$, $I\sqsubseteq^+ I'$ means that, for all but finitely many $n<\omega$, $I'_n$ is the union of some $I_k$'s. Notice that $\la I^+,\sqsubseteq^+\ra$ is a preorder but not directed, even more, $\cof(\Ior,\sqsubseteq^+)=\cfrak$ (see~\cite[Prop.~3.7]{CMS}). 

\begin{fact}\label{fac:b1}
For $I, I'\in\Ior$, if $I\sqsubseteq^+ I'$ then $(J,y)\sqsubset^I(J',y')$ implies $(J,y)\sqsubset^{I'}(J',y')$, In particular, $\la\Ior\times \cantor,\sqsubset^{I'}\ra \leqT \la\Ior\times \cantor,\sqsubset^I\ra$ and $\Mwf_{I'}\leqT \Mwf_I$. 
\end{fact}

Recall that $I^1$ denotes the partition of $\omega$ into singletons.

\begin{lemma}\label{lem:b3}
\ 
\begin{enumerate}[label=\rm(\alph*)]
    \item\label{lem:b3:1} $\Mwf_{I^1}=\Mwf$.
    \item\label{lem:b3:2} For all $I\in\Ior$, $\Mwf_I\subseteq\Mwf$ and $\Mwf_I\leqT\Mwf$.  
\end{enumerate} 
\end{lemma}
\begin{proof}
We just prove~\ref{lem:b3:1} because~\ref{lem:b3:2} is consequence of~\ref{lem:b3:1}, \autoref{fac:b0} and~\ref{fac:b1}. 
For $x\in\cantor$, $x\in A_{J,y}^{I^1}$ iff $\forall^\infty n\,\exists k\in J_n\colon x(k)=y(k)$, which is equivalent to $\forall^\infty n\colon x{\restriction}J_n\neq y'{\restriction}J_n$ where $y':=y+1 \mod{2}$. This actually shows $A_{J,y}^{I^1}=B_{y',J}$ (see~\autoref{def:b0}). Therefore, $\Mwf_{I^1}=\Mwf$.
\end{proof}

The previous representation of $\Mwf$ can be used to prove:

\begin{theorem}[{\cite{bartJudah}}]\label{cor:b4}
 $\MAwf=\bigcap_{I\in\Ior}\Mwf_I$.   
\end{theorem}
\begin{proof}
For $X\subseteq\cantor$, 
\begin{align*}
X\in\MAwf & \text{\ iff\ }  \forall A\in\Mwf\colon X+A\in\Mwf \\
 & \text{\ iff\ }\forall (I,y)\in\Ior\times\cantor\,\exists (J,y')\in\Ior\times \cantor\colon X+A_{I,y}^{I^1}=\underbrace{\bigcup_{x\in X}x+A_{I,y}^{I^1}}_{\bigcup_{x\in X}A_{I,x+y}^{I^1}}\subseteq A_{J, y'}^{I^1}\\
 & \text{\ iff\ } \forall (I,y)\in\Ior\times\cantor\,\exists (J,y')\in\Ior\times  \cantor\,\forall x\in X\colon A_{I,x+y}^{I^1}\subseteq A_{J,y'}^{I^1}\\
 & \text{\ iff\ } \forall I\in\Ior\,\exists(J,y)\in\Ior\times\cantor\,\forall x\in X\colon A_{I,x}^{I^1}\subseteq A_{J,y}^{I^1} \text{ (using a translation)}\\
 & \text{\ iff\ } \forall I\in\Ior\,\exists(J,y)\in\Ior\times\cantor\,\forall x\in X\,\forall^\infty n\in\omega\, \exists\ell<\omega\colon I_\ell\subseteq J_n\,\wedge\,x{\restriction}I_\ell=y{\restriction}I_\ell\\
 & \text{\ iff\ } \forall I\in\Ior\,\exists(J,y)\in\Ior\times\cantor\colon X\subseteq A_{J,y}^I\\
& \text{\ iff\ } \forall I\in\Ior\colon X\in\Mwf_I\\
& \text{\ iff\ } X\in\bigcap_{I\in\Ior}\Mwf_I. \qedhere
 \end{align*}
\end{proof}

We now look at the cardinal characteristics associated with $\vec\Mwf$ and its influence on those with $\MAwf$ as well. Since $\vec\Mwf$ is a directed scheme of ideals and $\MAwf=\vec\Mwf^*$, by \autoref{lem:c0} and~\ref{thm:c1} we immediately obtain:

\begin{lemma}\label{lem:b6}
Let $I, J\in \Ior$ and $D\subseteq \Ior$ cofinal. 
    \begin{enumerate}[label=\rm(\alph*)]
        \item\label{lem:b6:1} If $I\sqsubseteq J$ then $\Cv_\Mwf\leqT\Cv_{\Mwf_I}\leqT\Cv_{\Mwf_J}\leqT\Cv_{\MAwf} \leqT \prod_{I\in D}\Cv_{\Mwf_I}$. In particular,  $\cov(\Mwf)\leq \cov(\Mwf_I)\leq\cov(\Mwf_J)\leq \supcov(\vec\Mwf)\leq\cov(\MAwf) \leq \dfrak\left(\prod_{I\in D}\Cv_{\Mwf_I}\right)$ and $\non(\MAwf)=\minnon(\vec\Mwf)\leq\non(\Mwf_I)\leq \non(\Mwf_J)\leq\non(\Mwf)$.
        \item\label{lem:b6:2} $\MAwf \leqT \prod_{I\in D}\Mwf_I$. 
    \end{enumerate}
\end{lemma}

It is also clear that $\mincov(\vec\Mwf)=\cov(\Mwf)$ and $\supnon(\vec\Mwf) = \non(\Mwf)$. Although \autoref{lem:b6}~\ref{lem:b6:2} gives us some information about the additivity and cofinality numbers, we are going to prove much more.
Towards this direction, we first show a connection between $\baire$ and $\Mwf_I$.

\begin{lemma}\label{lem:b4}
 $\baire\leqT\Mwf_I$.  
\end{lemma}
\begin{proof}
It suffices to show that $\Ior\leqT \la \Ior\times \cantor, \sqsubset^I\ra$ because $\Ior\eqT\baire$ (see~\autoref{ex:c1}). Define the functions $\Psi_-\colon \Ior\to\Ior\times\cantor$ and $\Psi_+\colon \Ior\times\cantor \to \Ior$ as follows. 
For $J\in\Ior$, pick some $J^+\in\Ior$, coarser that $I$, such that $J\sqsubseteq J^+$. 
 
So define
\begin{align*}
    \Psi_-(J) & := (J^+,\bar0),\\
    \Psi_+(J',y') & := J',
\end{align*}
where $\bar0$ is the constant sequence with value $0$.

The pair $(\Psi_-,\Psi_+)$ is the desired Tukey connection: Assume $(J^+,\bar{0})\sqsubset^I(J',y')$. Since $J^+$ is coarser than $I$, we obtain $J^+\sqsubseteq J'$, so $J\sqsubseteq J'$.
\end{proof}

Another relevant result is the following.

\begin{lemma}\label{lem:b7}
    For any $I, I'\in\Ior$, $\Cv_{\Mwf_I}\leqT \Mwf_{I'}$. 
\end{lemma}
\begin{proof}
Find $I''\in\Ior$ coarser than $I'$ such that $I\sqsubseteq I''$. Using \autoref{fac:b1} and \autoref{lem:b6}~\ref{lem:b6:1}, we get $\Cv_{\Mwf_I}\leqT \Cv_{\Mwf_{I''}}\leqT \Mwf_{I''}\leqT \Mwf_{I'}$
\end{proof}

We are ready to prove~\autoref{thm:a4}. Since $\minnon(\vec\Mwf) = \non(\MAwf)$, it is already known from \cite[Lem.~2.3]{paw85} that $\add(\Mwf)= \min\{\bfrak,\minnon(\vec\Mwf)\}$.

\begin{theorem}\label{thm:b1}
For all $I\in\Ior$,
\begin{center}
\begin{tabular}{ccccc}
    $\add(\Mwf_I)$ & $=$ &  $\add(\Mwf)$ & $=$ & $\min\{\bfrak,\minnon(\vec\Mwf)\}$,\\
    $\cof(\Mwf_I)$ & $=$ &  $\cof(\Mwf)$ & $=$ & $\max\{\dfrak,\supcov(\vec\Mwf)\}$.
\end{tabular}
\end{center}
In terms of Tukey connections, $\baire\times\bigvee_{J\in\Ior}\Cv_{\Mwf_J} \leqT \Mwf_I \leqT \Mwf \leqT \left(\Ior;\bigvee_{J\in\Ior}\Cv_{\Mwf_J}\right)$. 
\end{theorem}
\begin{proof}

Thanks to \autoref{products}, it is enough to prove the Tukey connections. Since we already have $\Mwf_I\leqT\Mwf$ by \autoref{lem:b3}, the following inequalities remain.

$\baire\times\bigvee_{J\in\Ior}\Cv_{\Mwf_J} \leqT \Mwf_I$: By \autoref{lem:b7} and \autoref{products}, $\bigvee_{J\in\Ior}\Cv_{\Mwf_J} \leqT \Mwf_I$. On the other hand, $\baire\leqT \Mwf_I$ by \autoref{lem:b4}, so we can use \autoref{a255} to conclude the desired inequality. 

$\Mwf\leqT\left(\Ior;\bigvee_{J\in\Ior}\Cv_{\Mwf_J}\right)$: 
We can use $\la\Ior\times\cantor,\sqsubset^{I^1}\ra$ instead of $\Mwf$. 
Define $\Psi_-\colon \Ior\times \cantor \to \Ior\times {}^\Ior (\Ior\times\cantor)$ by $\Psi_-(I,y):=(I,\bar y)$, where $\bar y\colon \Ior\to \Ior\times\cantor$, $J\mapsto (J,y)$, and define $\Psi_+\colon \Ior\times\prod_{J\in\Ior}\Mwf_J \to \Ior\times \cantor$ by $\Psi_+(J,\bar B):= (J',y')$, where $(J',y')$ is chosen such that $B_J\subseteq A^J_{J',y'}$. This is the required Tukey connection: if $I\sqsubseteq J$ and $\bar y(J) \sqsubset^\vee \bar B$, i.e.\ $y\in B_J$, then $y\in A^J_{J',y'}$, so $(I,y)\sqsubset^{I^1} (J',y')$.
\end{proof}

As a consequence, we obtain the following inequalities (although $\add(\Mwf)\leq \add(\MAwf)$ is easy to prove directly). 

\begin{corollary}\label{cor:b80}
    For any dominating $D\subseteq\Ior$, 
    \[\MAwf \leqT \prod_{I\in D}\Mwf_I \leqT \Mwf^D.\] 
    In particular, $\MAwf\leqT \Mwf^\dfrak \leqT \Cv_{[\cof(\Mwf)]^{<\add(\Mwf)}}^\dfrak$, $\add(\Mwf)\leq \add(\MAwf)$ and $\cof(\MAwf) \leq \dfrak\left(\Mwf^\dfrak\right) \leq \cov\left(([\cof(\Mwf)]^{<\add(\Mwf)})^\dfrak\right)$. 
\end{corollary}

Next, we show how $\Mwf_I$ is related with the measure zero ideal. From now on, $\Lb$ denotes the Lebesgue measure on $\cantor$.

\begin{lemma}\label{lem:b5}
For $I\in\Ior$, the following statements are equivalent.
\begin{multicols}{3}
\begin{enumerate}[label=\normalfont (\roman*)]
    \item\label{b5i} $\Mwf_I\subseteq\Nwf$.
    \item\label{b5ii} $\Mwf_I\subseteq\Ewf$. 
    \item\label{b5iii} $\sum_{k<\omega}2^{-|I_k|}<\infty$.
\end{enumerate}
\end{multicols}
    
\end{lemma}
\begin{proof}
Since $\Mwf_I$ is generated by $F_\sigma$ sets,~\ref{b5i}${}\Leftrightarrow{}$\ref{b5ii} is clear. So 
it suffices to prove~\ref{b5i}${}\Leftrightarrow{}$\ref{b5iii}. Notice that \[A_{J,y}^I=\bigcup_{m<\omega}\bigcap_{n\geq m}\bigcup_{\substack{k<\omega\\ I_k\subseteq J_n}}[y{\restriction}I_k].\]   
$``\Leftarrow"$ Assume that $\sum_{k<\omega}2^{-|I_k|}<\infty$. Let $X\in\Mwf_I$, so there are $J\in\Ior$ and $y\in\cantor$ such that $X\subseteq A_{J,y}^I$. To see that $X\in\Nwf$, it is sufficient to show that 
\[\Lb\left(\bigcup_{m<\omega}\bigcap_{n\geq m}\bigcup_{\substack{k<\omega\\ I_k\subseteq J_n}}[y{\restriction}I_k]\right)=0.\] 
Observe that 
\[\Lb\left(\bigcup_{\substack{k<\omega\\ I_k\subseteq J_n}}[y{\restriction}I_k]\right)\leq \sum_{\substack{k<\omega\\ I_k\subseteq J_n}}2^{-|I_k|}.\]
Since $\sum_{k<\omega}2^{-|I_k|}<\infty$, $\bigcap_{n\geq m}\bigcup_{\substack{k<\omega\\ I_k\subseteq J_n}}[y{\restriction}I_k]$ has measure zero. Therefore, $A_{J,y}^I\in\Nwf$. 

$``\Rightarrow"$ Assume that $\sum_{k<\omega}2^{-|I_k|}=\infty$. Let $J\in\Ior$ be coarser than $I$ such that \[\sum_{\substack{k<\omega\\ I_k\subseteq J_n}}2^{-|I_k|}\geq n.\] We prove that $A_{J,y}^I\notin \Nwf$, in fact $\Lb(\cantor\smallsetminus A_{J,y}^I)=0$.
Notice that 
\[\cantor\smallsetminus A_{J,y}^I=\bigcap_{m<\omega}\bigcup_{n\geq m}\bigcap_{\substack{k<\omega\\ I_k\subseteq J_n}}(\cantor\menos[y{\restriction}I_k]).\]
To show $\Lb(\cantor\smallsetminus A_{J,y}^I)=0$, it is enough to see that
\[\lim_{m\to\infty}\Lb\left(\bigcup_{n\geq m}\bigcap_{\substack{k<\omega\\ I_k\subseteq J_n}}\cantor\smallsetminus[y{\restriction}I_k]\right)=0.\]
Indeed,
\begin{align*}
\Lb\left(\bigcap_{\substack{k<\omega\\ I_k\subseteq J_n}}\cantor\smallsetminus[y{\restriction}I_k]\right)&=\prod_{\substack{k<\omega\\ I_k\subseteq J_n}}(1-2^{-|I_n|})\\
&\leq e^{-\sum_{\substack{k<\omega\\ I_k\subseteq J_n}}2^{-|I_k|}}\\
&\leq e^{-n}
\end{align*}
Since $\sum_{n<\omega} e^{-n}<\infty$, we conclude the desired limit. 
\end{proof}

The previous proof allows to say a bit more in the case $\sum_{k<\omega}2^{-|I_k|}=\infty$.

\begin{lemma}\label{lem:b20}
If $\sum_{k<\omega}2^{-|I_k|}=\infty$ then $\Cv_\Nwf\leqT\Cv_{\Mwf_I}^\perp$. In particular, $\cov(\Nwf)\leq\non(\Mwf_I)$ and $\cov(\Mwf_I)\leq\non(\Nwf)$.   
\end{lemma}
\begin{proof}
    In the proof of \autoref{lem:b5}, we produced a set $X\in\Mwf_I$ such that $\cantor\menos X\in\Nwf$. Since both ideals $\Mwf_I$ and $\Nwf$ are translation invariant, we obtain $\Cv_\Nwf\leqT\Cv_{\Mwf_I}^\perp$ as in Rothberger's proof of $\Cv_\Nwf\leqT \Cv_{\Mwf}^\perp$, namely, using the maps $x\mapsto x+X$ and $y\mapsto y+(\cantor\menos X)$.
\end{proof}

Using~\autoref{lem:b5}, and~\ref{lem:b20}, \autoref{thm:a6} is concluded. 

We finish this section with connections between our ideals and the relational systems introduced in~\cite{CMR2} to study $\Cv_{\MAwf}$.

\begin{definition}[{\cite[Def.~2.7]{CMR2}}]\label{def:b3} Fix $b\in\baire$.
\begin{enumerate}[label=(\arabic*)]
\item\label{def:b3:1} For  $I\in\Ior$ and $f, h\in\baire$, define 
    \[f  \sqrb (I,h)\textrm{\ iff\ }\forall^\infty n\in\omega\,\exists k\in I_n\colon f(k)=h(k).\]
    \item\label{def:b3:2} Define the relational system $\Rbf_b:=\la\prod b,\Ior\times\prod b,\sqrb\ra$, where $\prod b:=\prod_{n<\omega}b(n)$.
\end{enumerate}
In the context of $\Rbf_b$, we will always consider that $b(n)>0$ for all $n$, even if we just write ``$b\in\baire$".
\end{definition}

Notice that, for fixed $(I,h)\in \Ior\times \prod b$, $\set{f\in\prod b}{f\sqrb (I,h)}$ is meager whenever $b\geq^* 2$, so $\Cv_\Mwf\leqT\Rbf_b$, which implies $\bfrak(\Rbf_b)\leq\non(\Mwf)$ and $\cov(\Mwf)\leq \dfrak(\Rbf_b)$.\footnote{In~\cite{cardona25}, these are denoted by $\bfrak_{b}^\mathsf{eq}$ and $\dfrak_{b}^\mathsf{eq}$, respectively.} On the other hand, if $b\ngeq^*2$ then we can find some $(I,h)\in \Ior\times \prod b$ such that $f\sqrb (I,h)$ for all $f\in\prod b$, so $\dfrak(\Rbf_b) = 1$ and $\bfrak(\Rbf_b)$ is undefined.

\begin{fact}\label{fct:b2}
For $b\in\omega^\omega$, 
$\Rbf_b\eqT\la\prod b,\Ior\times\baire,\sqrb\ra$. 
As a consequence, if $b'\in\omega^\omega$ and
$b\leq^* b'$, then $\Rbf_b\leqT\Rbf_{b'}$. In particular, $\bfrak(\Rbf_{b'}) \leq \bfrak(\Rbf_b)$ and $\dfrak(\Rbf_b)\leq \dfrak(\Rbf_{b'})$.   
\end{fact}

In~\cite{CMR2} we have proved that $\Rbf_b\leqT \Cv_{\MAwf} \leqT \prod_{b\in D}\Rbf_b$ for any $b\in\baire$ and any dominating $D\subseteq\baire$, and also that $\dfrak(\Rbf_b)\leq\cof(\Mwf)$. These facts can be deduced from the following result that connects $\Rbf_{b}$ and $\Cv_{\Mwf_I}$.

\begin{lemma}\label{lem:b8}
    For $I\in\Ior$, if $\forall^ \infty k\colon b(k)=2^{|I_k|}$ then $\Rbf_{b}\eqT\Cv_{\Mwf_I}$.
\end{lemma}
\begin{proof}
Assume wlog that, for all $k\in\omega$, $b(k)=2^{|I_k|}$, so $\prod b=2^I := \prod_{k<\omega}{}^{I_k}2$. 
Define 
\begin{align*}
f\colon 2^I & \longrightarrow \cantor \\
        \eta & \longmapsto  \bigcup_{n<\omega}\eta(n)
\end{align*}
Note that $f$ is a bijection. 

Let $\Ior_I$ be the set of $J\in\Ior$ coarser than $I$. 
Define the maps
\begin{align*}
g  \colon \Ior_I  &
                \longrightarrow \Ior \text{ and}                     \notag\\
h  \colon \Ior\times 2^I & \longrightarrow\Ior_I\times\cantor \notag
\intertext{by the assignments}
J & \longmapsto \seq{\set{k<\omega}{I_k\subseteq J_n}}{n<\omega} \text{ and} \\
(J,y) &\longmapsto \big(g^{-1}(J),f(y)\big),
\end{align*}
respectively. These maps are bijections as well.

To define $\Psi_+\colon\Mwf_I\to\Ior\times2^I$, for $X\in \Mwf_I$, pick an $(J,y)\in\Ior_I\times\cantor$ such that $X\subseteq A_{J,y}^I$ and put $\Psi_+(X):=h^{-1}(J,y)$. Notice that $f(\eta)\in A_{J,y}^I$ iff $\forall^\infty n\,\exists k\colon I_k\subseteq J_n$ and $\eta(k)=y{\restriction}I_k$, which is equivalent to $\forall^\infty n\,\exists k\in J_n'\colon\eta(k)=y{\restriction}I_k$ where $J'=g(J)$, i.e.\ $\eta\sqrb\Psi_+(X)$. Then $(f,\Psi_+)$ witnesses $\Rbf_b\leqT\Cv_{\Mwf_I}$.

To prove $\Cv_{\Mwf_I}\leqT\Rbf_b$, define 
\begin{align*}
\Psi'_+\colon\Ior\times2^I & \rightarrow \Mwf_I \\
        (J',\eta) & \mapsto A_{(g^{-1}(J'),f(\eta))}^I
\end{align*}
It is clear that $(f^{-1},\Psi'_+)$ witnesses $\Cv_{\Mwf_I}\leqT\Rbf_b$. Indeed, $f^{-1}(x)\sqrb(J',\eta)$ iff $\forall^{\infty}n\,\exists k\in J'_n\colon x{\restriction} I_k=\eta(k)$, which is equivalent to $\forall^{\infty}n\,\exists k\colon I_k\subseteq J_n$ and $x{\restriction}I_k=\eta(k)$ , i.e.\ $x\in A_{(g^{-1}(J'),f(\eta))}^I$. 
\end{proof}

Because of the above result, we obtain:

\begin{corollary}
\ 
    \begin{enumerate}[label=\rm(\arabic*)]
        \item If $\forall^ \infty k\colon \log_2 b(k)\leq|I_k|$, then $\Rbf_{b}\leqT\Cv_{\Mwf_I}$.
        \item If $\forall^ \infty k\colon |I_k|\leq\log_2 b(k)$, then $\Cv_{\Mwf_I}\leqT\Rbf_{b}$.
    \end{enumerate}
\end{corollary}

\begin{corollary}
    $\min_{b\in\baire}\bfrak(\Rbf_b)=\minnon(\vec\Mwf)$ and $\sup_{b\in\baire}\dfrak(\Rbf_b)=\supcov(\vec\Mwf)$.\footnote{The equality $\min_{b\in\baire}\bfrak(\Rbf_b)=\non(\MAwf)$ is originally due to Bartoszy\'nski and Judah~\cite{bartJudah}.}
\end{corollary}

\section{Consistency results}\label{sec:s4}

In this section, we prove our consistency results, i.e.\ \autoref{thm:a11}-\ref{thm:a99}. We start reviewing some notation.

\begin{definition}\label{def:a0}
Given a sequence of non-empty sets $b = \seq{b(n)}{n\in\omega}$ and $h\colon \omega\to\omega$, define
\begin{align*}
 \prod b &:= \prod_{n\in\omega}b(n),\textrm{\ and} \\
 \Swf(b,h) &:= \prod_{n\in\omega} [b(n)]^{\leq h(n)}.
\end{align*}
For two functions $x\in\prod b$ and $\varphi\in\Swf(b,h)$ write  
\[x\,\in^*\varphi\textrm{\ iff\ }\forall^\infty n\in\omega:x(n)\in \varphi(n).\]

Let $\Lc(b,h) := \la\prod b, \Scal(b,h),\in^*\ra$, which is a relational system. Denote $\blc_{b,h}=\bfrak(\Lc(b,h))$ and $\dlc_{b,h}=\dfrak(\Lc(b,h))$.\footnote{In~\cite{KM,CKM}, these are denoted by $\vfa_{b,h}^\forall$ and $\cfrak_{b,h}^\forall$, respectively.} Also define
\[\minLc:=\min\set{\blc_{b,\id}}{b\in\baire}\lay\supLc:= \sup\set{\dlc_{b,\id}}{b\in\baire}.\]
\end{definition}

Recall the following characterization. 
Pawlikowski~\cite{paw85} characterized $\add(\Nwf)$, while $\cof(\Nwf)$ is due to the first and second author~\cite[Lem.~3.11]{CM}.

\begin{theorem}\label{thm:a13}
$\add(\Nwf)=\min\{\bfrak,\minLc\}$ and $\cof(\Nwf)=\max\{\dfrak,\supLc\}$.
\end{theorem}

Below, we strengthen the characterization of $\NAwf$ from \autoref{thm:a0}~\ref{thm:a0:NA}.

\begin{theorem}\label{thm:b20}
    Let $X\subseteq \cantor$ and $D\subseteq \Ior$ a dominating family. Then
    $X\in\NAwf$ iff, for all $I\in D$, there is some $\varphi\in\prod_{n\in \omega}\pts({}^{I_n}2)$ such that $\forall n\in \omega\colon |\varphi(n)|\leq n$ and $X\subseteq H_\varphi$.
\end{theorem}
\begin{proof}
    Under \autoref{thm:a0}~\ref{thm:a0:NA}, we only need to prove ``$\Leftarrow$'', that is, the statement about $D$ implies the same statement but for $\Ior$ in the place of $D$. 

    Let $J\in\Ior$ be such that $|J_n|=(n+1)^2$ for all $n<\omega$, and fix $I\in \Ior$.  Set $I_n':=\bigcup_{k\in J_n}I_k$, so $I'\in\Ior$. 
    Since $D$ is $\Ior$-dominating, we can find an $I^*\in D$ such that, for all but finitely many $n<\omega$, $I^*$ contains at least two intervals from $I'$. Therefore, there is some $\varphi^*\in\Swf({}^{I^*}2,\id)$ such that $X\subseteq H_{\varphi^*}$. It is enough to construct a $\varphi\in\Swf({}^I2,\id)$ such that $H_{\varphi^*}\subseteq H_\varphi$.

    We first show that, for all but finitely many $n<\omega$ and for any $k<\omega$, if $I_k\subseteq I^*_n$ then $k\geq n(n+1)$: For any $\ell<\omega$, $I'_\ell$ is composed of $(\ell+1)^2$-many intervals from $I$, so $I_k\subseteq I'_\ell$ iff $\frac{\ell(\ell+1)(2\ell+1)}{6} \leq k <\frac{(\ell+1)(\ell+2)(2\ell+3)}{6}$. Now, for all but finitely many $n$, we have that whenever $I'_\ell$ intersects $I^*_n$, $\ell\geq n$ and $\ell(\ell+1)< \frac{(\ell-1)\ell(2\ell-1)}{6}$. Therefore, for any $I_k\subseteq I^*_n$, $n(n+1)\leq \ell'(\ell'+1)< \frac{(\ell'-1)\ell'(2\ell'-1)}{6} \leq k$, where $\ell'$ is the smallest $\ell$ such that $I'_\ell\cap I^*_n\neq\emptyset$. 

    Fix an $n$ as above. Whenever $I_k\subseteq I^*_n$, define $\varphi(k):= \set{t{\restriction}I_k}{t\in\varphi^*(n)}$, so $|\varphi(k)|\leq|\varphi^*(n)|\leq n\leq n(n+1) \leq k$; and whenever $I_k$ intersects both $I^*_n$ and $I^*_{n+1}$, define
    \[\varphi(k):=\set{s{\restriction}(I_k\cap I_n^*)\cup t{\restriction}(I_{k}\cap I_{n+1}^*)}{s\in\varphi^*(n), t\in\varphi^*(n+1)},\]
    so $|\varphi(k)|\leq|\varphi^*(n)|\cdot|\varphi^*(n+1)|\leq n(n+1)\leq k$. 

    So far, we have defined $\varphi(k)$ for all but finitely many $k<\omega$, so for the remaining ones set $\varphi(k):=\emptyset$. Therefore, $\varphi\in\Swf({}^I2,\id)$.

    Lastly, it is clear that $H_{\varphi^*}\subseteq H_{\varphi}$, as required.
\end{proof}

Thanks to the above, we have the following convenient Tukey connection.

\begin{theorem}\label{thm:b3}
For any dominating family $D\subseteq\omega^\omega$, we have $\Cv_{\NAwf}\leqT\prod_{b\in D}\Lc(b,\id)$.
\end{theorem}
\begin{proof}
 Without loss of generality, we may assume that there is some $\Ior$-dominating family $D_0$ such that, for each $b\in D$, there is some $I\in D_0$ such that $b = {}^{I}2$, i.e.\ $b(n) = {}^{I_n}2$ for all $n<\omega$. 

We define $\Psi_-\colon \cantor\to\prod_{I\in D_0}{}^{I}2$ and
$\Psi_+\colon\prod_{I\in D_0}\Swf({}^{I}2,\id)\to \NAwf$ as follows: For $x\in\cantor$ and $I\in D_0$, set $x_I:=\seq{x{\restriction}I_n}{n\in\omega}$ and define $\Psi_-(x):=\seq{x_I}{I\in D_0}$; 
for $\bar\varphi=\seq{\varphi_I}{I\in D_0}\in\prod_{I\in D_0}\Swf({}^{I}2,\id)$ define \[\Psi_+(\bar\varphi):=\bigcap_{I\in D_0}\set{x\in\cantor}{\forall^\infty n\in\omega\colon x{\restriction}I_n\in\varphi_I(n)}.\]
Notice that $\Psi_+(\bar\varphi)\in\NAwf$ by \autoref{thm:b20}. It is clear that, for any $x\in\cantor$ and $\bar\varphi=\seq{\varphi_I}{I\in D_0}\in\prod_{I\in D_0}\Swf({}^{I}2,\id)$, $\Psi_-(x)\sqsubset^\otimes \bar\varphi$ implies $x\in\Psi_+(\bar\varphi)$. 
\end{proof}

We are ready to prove our consistency results. First, we review some terminology.

\begin{definition}\label{defposet} Define the following forcing notions 
\begin{enumerate}[label=\normalfont(\arabic*)]
    \item  Fix $b$ and $h$ as in~\autoref{def:a0}. Define $\Loc_{b,h}$ as the poset whose conditions are pairs $p=(\varphi_p,n_p)$ where $\varphi_p\in\Swf(b,h)$ and $n_p<\omega$ such that, for some $m_p<\omega$, $|\varphi_p(i)|\leq m_p$ for all $i<\omega$. The order is defined by $q\leq p$ iff $n_p\leq n_q$, $\varphi_q(i)=\varphi_p(i)$ for all $i<n_p$ and $\forall i<\omega\colon \varphi_p(i)\subseteq \varphi_q(i)$.
  
    \item\label{defposet2} \emph{Hechler forcing} is defined by  $\Dor=\omega^{<\omega}\times\baire$,
ordered by $(t,g)\leq(s,f)$ if $s\subseteq t$, $f\leq g$ and $f(i)\leq t(i)$ for all $i\in |t|\menos|s|$. This forcing is used to increase $\bfrak$. Recall that $\Dor$ is $\sigma$-centered.

  \item For an infinite cardinal  $\theta$, $\Fn_{<\theta}(A,B)$ denotes the poset of partial functions from $A$ into $B$ of size ${<}\theta$, ordered by $\supseteq$.
\end{enumerate}
\end{definition}

\begin{theorem}\label{thm:c0}
 Let $\theta\leq\lambda$ and $\nu$ be uncountable cardinals such that $\theta$ is regular, $\lambda^{\aleph_0} = \lambda$ and $\dfrak_\theta=\nu$.\footnote{Notice that no inequality is assumed between $\nu$ and $\lambda$.} Then there is a ccc poset forcing 
\[
    \add(\Nwf) = \cof(\Nwf) = \theta \leq \cov(\NAwf) \leq \nu \text{ and } \cfrak =\lambda.
\]
In particular, it is consistent with $\thzfc$ that $\cov(\NAwf) < \lambda$ (starting with $\nu<\lambda$).
\end{theorem}

Here, $\dfrak_\theta$ is the canonical dominating number of ${}^\theta \theta$, which coincides with $\dfrak(\la\theta,\theta,\leq\ra^\theta)$. It is easy to force $\dfrak_\theta = \nu$, e.g.\ with $\Fn_{<\theta}(\nu,2)$ (by assuming $\theta^{<\theta}=\theta$ and $\nu^\theta = \nu$ in the ground model). 

\begin{proof}
Perform a FS iteration $\Por = \seq{\Por_\xi,\Qnm_\xi}{\xi<\lambda\theta}$ ($\lambda\theta$ as product of ordinal numbers) where $\Qor_\xi$ is a $\Por_\xi$-name of $\Dor\ast \Loc_{\dot d_\xi, \id}$ where $\dot d_\xi$ is the name of the dominating real over $V_\xi := V^{\Por_\xi}$ added by $\Dor$. Using the cofinaly-many Cohen and dominating reals $\seq{\dot d_{\lambda\rho}}{\rho<\theta}$ added along the iteration and $\cf(\lambda\theta)=\theta$, we obtain $\bfrak = \non(\Mwf) = \cov(\Mwf) = \dfrak = \theta$. which implies $\add(\Mwf) = \cof(\Mwf)=\theta$. Even more, we obtain $\omega^\omega \eqT\Cv_\Mwf \eqT \theta$.

In the final generic extension $V_{\lambda\theta}$, it is clear that $D:=\set{d_{\lambda\rho}}{\rho<\theta}$ is $\leq^*$-increasing and dominating in $\baire$. Denote $d'_\rho := d_{\lambda\rho}$ for $\rho<\theta$. We show that $\Lc(d'_\rho,\id) \eqT \theta$. On the one hand, $\theta \eqT\Cv_\Mwf \leqT \Lc(d'_\rho,\id)$. For the converse, define $F\colon \prod d'_\rho \to \theta$ such that, for $x\in \prod d'_\rho$, $F(x)$ is some ordinal $\eta>\rho$ such that $x\in V_{\lambda\eta}$; and define $F'\colon \theta\to \Swf(d'_\rho,\id)$ such that $F'(\varrho)$ is the $\Loc_{d'_\varrho, \id}$-generic real added by $\Qor_{\lambda\varrho}$ when $\varrho\geq\rho$, otherwise $F'(\varrho):= F'(\rho)$. It is clear that $(F,F')$ is the desired Tukey connection. 

Since $\Lc(d'_\rho,\id) \eqT \theta$, we get $\minLc=\supLc=\theta$, so by \autoref{thm:a13} $\add(\Nwf)=\cof(\Nwf)=\theta$. On the other hand, given that $\Por$ is ccc, the equality $\dfrak_\theta = \nu$ is preserved (see e.g.~\cite[Lem.~6.6]{CarMej23}). Now, by \autoref{thm:b3},
\[\Cv_{\NAwf} \leqT \prod_{b\in D}\Lc(b,\id) \eqT \la\theta,\leq \ra^\theta,\]
so $\cov(\NAwf) \leq \dfrak_\theta = \nu$.
\end{proof}

Finally, we show that $\cof(\MAwf)<\non(\SNwf)$ is valid in our model from~\cite[Thm.~4.2]{CMR2}.

\begin{theorem}
     Let $\theta\leq\lambda$ and $\nu$ be uncountable cardinals such that $\theta$ is regular, $\dfrak_\theta = \nu$ and $\lambda^{\aleph_0} = \lambda$. Then there is a ccc poset forcing 
\[
     \cov(\Nwf) = \aleph_1  \leq \add(\Mwf) = \cof(\Mwf) = \theta 
      \leq \non(\Nwf) =\cfrak =\lambda \text{ and }\cof(\MAwf)\leq \nu.
\]
In particular, it is consistent with $\thzfc$ that $\cof(\MAwf) < \non(\Nwf)$.
\end{theorem}
\begin{proof}
    Let $\Por$ be the FS iteration of ccc posets constructed in the proof of~\cite[Thm.~4.2]{CMR2}. It remains to show that $\Por$ forces $\cof(\MAwf)\leq \nu$. Since $\Por$ has the ccc, the identity $\dfrak_\theta=\nu$ is preserved in the generic extension. Now, in the generic extension, by \autoref{cor:b80}, $\cof(\MAwf)\leq \cov\left(([\theta]^{<\theta})^\theta\right) = \dfrak_\theta=\nu$. 
\end{proof}

We still do not know how to control the values of $\cov(\NAwf)$, $\cov(\MAwf)$ and $\cof(\MAwf)$ in the previous models. In the case of the cofinality, we may need to develop some framework for lower bounds as we did in~\cite{CarMej23} for the Yorioka ideals. In particular, we ask:

\begin{question}
    Is it consistent with $\thzfc$ that $\cof(\MAwf)>\cfrak$?
\end{question}

Concerning $\NAwf$ we may ask:

\begin{question}
    Does $\thzfc$ proves an inequality between $\cof(\NAwf)$ and $\cfrak$?
\end{question}

We have some ideas to construct a directed scheme for $\NAwf$ to prove the consistency of $\cof(\NAwf)<\cfrak$, but we will develop them in the second part of this work.


{\small
\bibliography{name}
\bibliographystyle{alpha}
}


\end{document}